\newif\ifArXiV

\ArXiVtrue	

\ifArXiV
\documentclass[a4paper]{article}
\usepackage[affil-it]{authblk}
\usepackage[authoryear]{natbib}
\newenvironment{frontmatter}{}{}

\let\address\affil
\newenvironment{keyword}{\small \textbf{Keywords: }}{}
\else
\documentclass[authoryear,preprint]{elsarticle}     
\fi
\usepackage{fullpage}

\usepackage{graphicx}
\usepackage{pdflscape}
\usepackage{secdot}
\usepackage{subcaption}
\usepackage{amsmath, amssymb,amsthm}
\usepackage{amsfonts}
\usepackage{latexsym}
\usepackage{booktabs}
\usepackage{enumitem,blindtext}
\usepackage[export]{adjustbox}
\usepackage{mathrsfs}
\usepackage[usenames,dvipsnames]{color}
\usepackage[mathcal]{eucal}
\usepackage{caption}
\usepackage{multirow}
\usepackage{amsfonts}
\usepackage[linesnumbered, noend, ruled]{algorithm2e}
\usepackage{setspace}
\usepackage{xcolor}
\usepackage{algorithmicx}
\usepackage{algpseudocode}
\usepackage{todonotes}

\newcommand{\comments}[1]{}
\DeclareGraphicsExtensions{.png,.pdf,.jpg}
\graphicspath{{./Figures_revision/}{./Figures/} }

\newtheorem{theorem}{Theorem}
\newtheorem{definition}{Definition}

\newtheorem{proposition}{Proposition}

\usepackage{tikz}
\usepackage{tikzsymbols}
\usetikzlibrary{arrows,automata}
\usetikzlibrary{shapes.geometric, arrows.meta}
\usetikzlibrary{positioning}

\usepackage[normalem]{ulem}
\DeclareMathOperator*{\argmin}{arg\,min}

\usepackage{xstring}

\usepackage{hyperref}

\begin{document}

\begin{frontmatter}

\title{A matheuristic for solving the single row facility layout problem}

\ifArXiV
\author[1]{Thomas Pammer\thanks{tpammer@faw.jku.at}}
\author[2]{Markus Sinnl\thanks{markus.sinnl@jku.at}}
\date{}
\else
\author[1]{Thomas Pammer}
\ead{tpammer@faw.jku.at}
\author[2]{Markus Sinnl}
\ead{markus.sinnl@jku.at}
\fi

\address[1]{Institute for Application-Oriented Knowledge Processing (FAW), Johannes Kepler University Linz, Austria}
\address[2]{Institute of Business Analytics and Technology Transformation/JKU Business School, Johannes Kepler University Linz, Austria}
\ifArXiV
\maketitle	
\fi
\begin{abstract}
The single row facility layout problem (SRFLP) is a well-studied NP-hard 
combinatorial optimization problem with applications in manufacturing and 
logistics systems.
In the SRFLP, a set of facilities with lengths is given, as well as weights 
between each 
pair of facilities. The facilities must be arranged on a line, such that the sum 
of the weighted 
center-to-center distances is minimized.
In this work, we introduce a novel matheuristic approach that integrates exact 
optimization into a metaheuristic framework based on simulated annealing to 
effectively solve large-scale SRFLP instances.
Specifically, we propose the window approach matheuristic, which solves 
subsegments of the layout to optimality using mixed-integer programming while 
preserving the ordering of facilities outside the window. 
%
To the best of our knowledge, this constitutes the first matheuristic approach specifically designed for the SRFLP.
We evaluate the performance of our method on the widely-used benchmark instance 
sets from literature. The computational results
demonstrate that our matheuristic improves the best-known solution values for 17 
of 70 instances, and matches the best-known solution values for the remaining 53 
instances, outperforming current state-of-the-art metaheuristics.
\end{abstract}

\ifArXiV

\begin{keyword}
Matheuristic, Integer Programming, Multi-start Simulated Annealing, Facility Layout
\end{keyword}
\else

\begin{keyword}
Matheuristic \sep Integer Programming \sep Multi-start Simulated Annealing \sep Facility Layout
\end{keyword}
\fi

\end{frontmatter}

\section{Introduction}
The \emph{single row facility layout problem (SRFLP)}, is a widely studied and 
prominent variant of the 
\emph{linear ordering problem} \citep{marti2011linear}.
The SRFLP was first introduced by~\citet{simmons1969one}, is NP-hard~\citep{amaral2013polyhedral} and has a wide range of 
applications, including the arrangement of rooms in hospitals and 
offices~\citep{simmons1969one}, warehouse layout design~\citep{picard1981one}
and flexible manufacturing systems~\citep{heragu1988machine}.
The SRFLP is defined as follows~\citep{simmons1969one}.

\begin{definition}
An instance of the SRFLP is defined by a set of $n$ facilities, each with a 
positive integer length $\ell_i$ for $i \in \{1, \dots, n\}$,
and a symmetric matrix of positive integer weights $w_{ij}$, representing the interaction cost between each facility pair $(i,j)$.
A feasible solution is any permutation $\pi$ of the facilities placed sequentially on a one-dimensional line. The objective is to find a permutation $\pi^*$ that minimizes the total
weighted center-to-center distance between all facility pairs, which is defined as:
\begin{equation}
    F(\pi) = \min \quad \sum_{i=1}^{n-1}\sum_{j=i+1}^{n} w_{ij} d^{\pi_{ij}},
    \label{eq: objective_theory}
\end{equation}
where $d^{\pi_{ij}}$ denotes the center-to-center distance between facilities $i$ and $j$ in a given permutation $\pi$. The symmetry of the weight matrix allows limiting the summation to $j > i$.
\end{definition}

As originally stated by~\citet{simmons1969one}, the center-to-center distance between two facilities corresponds to the sum of half their respective lengths and the total
length of all facilities located between them. For any pair of indices $i < j$, 
this can be calculated as (see, e.g.,~\citep{palubeckis2017single}):
\begin{equation}
    d^{\pi_{ij}} = \frac{\ell_{\pi_i}}{2} + \sum_{i<k<j} \ell_{\pi_k} + \frac{\ell_{\pi_j}}{2},
    \label{eq: dist_theory}
\end{equation}
where $\pi_i$ refers to the facility placed at position $i$ in the current solution.

In figure \ref{fig:weights_and_visuals} an exemplary instance of the SRFLP is 
given together with an illustration of the computation of the objective function 
value (i.e., the weighted distances computation).

\begin{figure}[ht]
\centering

\begin{minipage}[t]{0.30\textwidth}
\centering
\begin{subfigure}[t]{\linewidth}
\begin{tabular}{cccccc}
\hline
\multicolumn{6}{c}{$\ell = (3,\ 2,\ 1,\ 4,\ 2,\ 1)$} \\
\hline
\vspace{0.2cm}\\
\hline
0 & 1 & 4 & 3 & 2 & 1 \\
1 & 0 & 3 & 1 & 3 & 2 \\
4 & 3 & 0 & 1 & 0 & 1 \\
3 & 1 & 1 & 0 & 2 & 2 \\
2 & 3 & 0 & 2 & 0 & 3 \\
1 & 2 & 1 & 2 & 3 & 4 \\
\hline
\end{tabular}
\caption*{Weights matrix}
\label{tab:weigths}
\caption{Instance description}
\end{subfigure}
\end{minipage}
\hfill
\begin{minipage}[t]{0.68\textwidth}
\centering
\begin{subfigure}[t]{\linewidth}
\centering
\begin{tikzpicture}[node distance= 0.2cm, >={Stealth[round]}, thick]
    \draw[draw=black] (0.5,0) rectangle ++(2.25, 0.75) node[pos=.5]{} ;
    \draw[draw=black] (2.75,0) rectangle ++(1.5, 0.75) node[pos=.5]{};
    \draw[draw=black] (4.25,0) rectangle ++(0.75, 0.75) node[pos=.5]{};
    \draw[draw=black] (5,0) rectangle ++(3, 0.75) node[pos=.5]{};
    \draw[draw=black] (8,0) rectangle ++(1.5, 0.75) node[pos=.5]{};
    \draw[draw=black] (9.5,0) rectangle ++(0.75, 0.75) node[pos=.5]{};
    \node[] at (1.625, 0.25) {1};
    \node[] at (3.5, 0.25) {2};
    \node[] at (4.625, 0.25) {3};
    \node[] at (6.5, 0.25) {4};
    \node[] at (8.75, 0.25) {5};
    \node[] at (9.875, 0.25) {6};
    \draw[-Latex, semithick] (1.625, 0.75) arc[radius= 1.35, start angle= 135, end angle=45];
    \draw[-Latex, semithick] (1.625, 0.75) arc[radius= 2.1213, start angle= 135, end angle=45];
    \draw[-Latex, semithick] (1.625, 0.75) arc[radius= 3.44715, start angle= 135, end angle=45];
    \draw[-Latex, semithick] (1.625, 0.75) arc[radius= 5.03775, start angle= 135, end angle=45];
    \draw[-Latex, semithick] (1.625, 0.75) arc[radius= 5.8335, start angle= 135, end angle=45];
    \node[] at (3.6, 1.1) {$2.5 {\cdot} 1$};
    \node[] at (4.6, 1.1) {$4 {\cdot} 4$};
    \node[] at (6.75, 1.1) {$6.5 {\cdot} 3$};
    \node[] at (8.95, 1.1) {$9.5 {\cdot} 2$};
    \node[] at (10, 1.1) {$11 {\cdot} 1$};
\end{tikzpicture}
\caption{Computation of weighted distances for all pairs containing the first facility}
\label{fig:comp_w_dist_1_fac}
\end{subfigure}

\vspace{0.5cm}

\begin{subfigure}[t]{\linewidth}
\centering
\begin{tikzpicture}[node distance= 0.2cm, >={Stealth[round]}, thick]
    \draw[draw=black] (0.5,0) rectangle ++(2.25, 0.75) node[pos=.5]{} ;
    \draw[draw=black] (2.75,0) rectangle ++(1.5, 0.75) node[pos=.5]{};
    \draw[draw=black] (4.25,0) rectangle ++(0.75, 0.75) node[pos=.5]{};
    \draw[draw=black] (5,0) rectangle ++(3, 0.75) node[pos=.5]{};
    \draw[draw=black] (8,0) rectangle ++(1.5, 0.75) node[pos=.5]{};
    \draw[draw=black] (9.5,0) rectangle ++(0.75, 0.75) node[pos=.5]{};
    
    \node[] at (1.625, 0.25) {1};
    \node[] at (3.5, 0.25) {2};
    \node[] at (4.625, 0.25) {3};
    \node[] at (6.5, 0.25) {4};
    \node[] at (8.75, 0.25) {5};
    \node[] at (9.875, 0.25) {6};

    \draw[-Latex, semithick] (1.625, 0.75) arc[radius= 1.35, start angle= 135, end angle=45];
    \draw[-Latex, semithick] (1.625, 0.75) arc[radius= 2.1213, start angle= 135, end angle=45];
    \draw[-Latex, semithick] (1.625, 0.75) arc[radius= 3.44715, start angle= 135, end angle=45];
    \draw[-Latex, semithick] (1.625, 0.75) arc[radius= 5.03775, start angle= 135, end angle=45];
    \draw[-Latex, semithick] (1.625, 0.75) arc[radius= 5.8335, start angle= 135, end angle=45];
    
    \draw[-Latex, semithick, green, dashed] (3.5, 0) arc[radius= 0.795495, start angle= 225, end angle=315];
    \draw[-Latex, semithick, green, dashed] (3.5, 0) arc[radius=2.1213, start angle= 225, end angle=315];
    \draw[-Latex, semithick, green, dashed] (3.5, 0) arc[radius= 3.712275, start angle= 225, end angle=315];
    \draw[-Latex, semithick, green, dashed] (3.5, 0) arc[radius= 4.5078, start angle= 225, end angle=315];
    \draw[-Latex, semithick, red, dotted] (4.625, 0.75) arc[radius= 1.325828, start angle= 135, end angle=45];
    \draw[-Latex, semithick, red, dotted] (4.625, 0.75) arc[radius= 2.91675, start angle= 135, end angle=45];
    \draw[-Latex, semithick, red, dotted] (4.625, 0.75) arc[radius= 3.712313, start angle= 135, end angle=45];
    \draw[-Latex, semithick, blue, dash pattern=on 8pt off 1pt] (6.5, 0) arc[radius= 1.59099, start angle= 225, end angle=315];
    \draw[-Latex, semithick, blue, dash pattern=on 8pt off 1pt] (6.5, 0) arc[radius= 2.386485, start angle= 225, end angle=315];
    \draw[-Latex, semithick, orange, dash pattern=on 16pt off 1pt] (8.75, 0.75) arc[radius= 0.795495, start angle= 135, end angle=45];
\end{tikzpicture}
\caption{Computation of weighted distances for all pairs of facilities}
\label{fig:comp_all_w_dist}
\end{subfigure}
\end{minipage}

\caption{Weights matrix and visual representation of total weighted distance computations}
\label{fig:weights_and_visuals}
\end{figure}
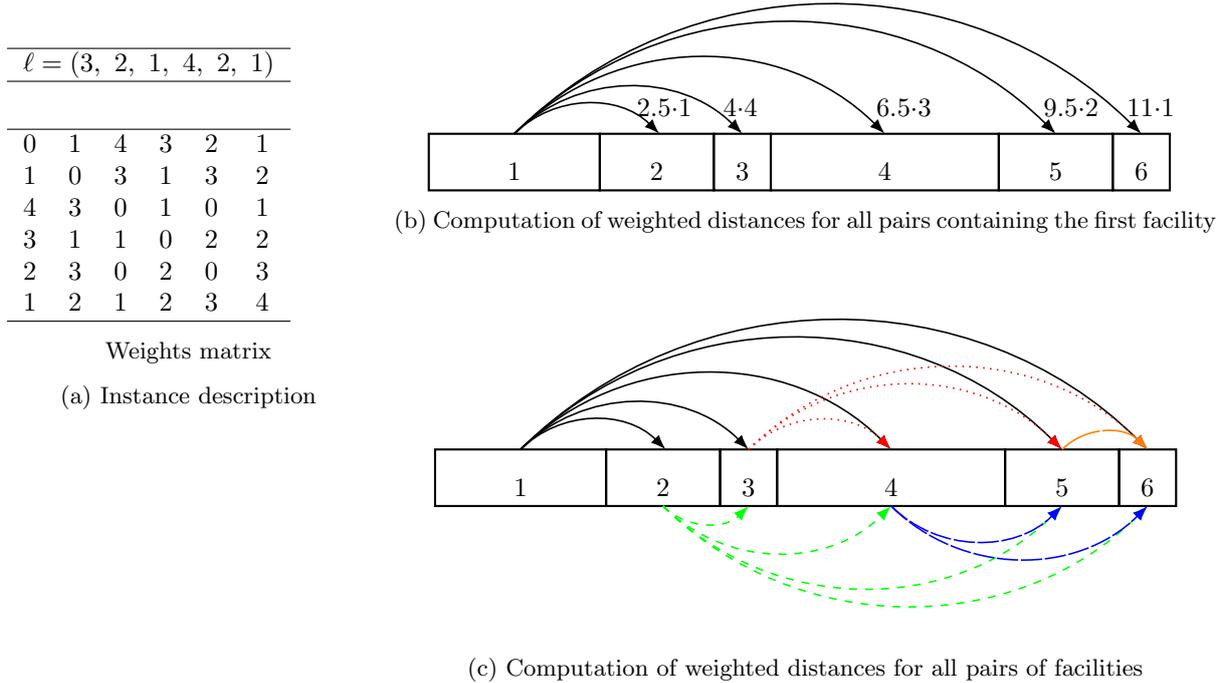

Figures~\ref{fig:comp_w_dist_1_fac} and~\ref{fig:comp_all_w_dist} depict the step-by-step calculation of the objective value,
initially by considering all pairs involving the first facility and then extending to all possible facility pairs.
In the illustrated example, the total weighted sum of all pairwise distances results in an objective value of 143.5.
The goal of the SRFLP is to identify the facility permutation that minimizes this total weighted distances.

\paragraph{Contribution and outline}

Given the high difficulty of the SRFLP, matheuristic approaches 
could represent a promising strategy for obtaining high-quality solutions within 
practical time constraints.
Matheuristic algorithms, which integrate metaheuristic strategies with exact 
mathematical programming components,
have recently gained traction for solving NP-hard problems 
\citep{maniezzo2021matheuristics}.
They combine the global exploration power of metaheuristics with the accuracy of 
exact methods applied to solve subproblems. Surprisingly, despite the SRFLP's 
relevance and difficulty,
no matheuristic approach has yet been proposed for this problem.
This work aims to fill this gap by developing a novel matheuristic algorithm that 
integrates exact optimization techniques into a metaheuristic framework,
thereby aiming to improve upon current state-of-the-art approaches for the problem.

Specifically, we integrate a \emph{mixed-integer programming (MIP)} model
into a metaheuristic framework by introducing a technique which we denote as the 
\emph{window approach}. This method allows us to solve selected subsections (i.e., 
windows) of the 
SRFLP layout to optimality,
while fixing the positions of the remaining facilities. This hybridization enables the exploitation of problem-specific structure without compromising computational efficiency.
Our algorithm starts with a \emph{multi-start simulated annealing (MSA)} phase 
that generates diverse and high-quality solutions. These solutions are 
subsequently refined through
two \emph{local search} procedures and the proposed window approach. The window 
approach preserves the global permutation outside the selected window,
while optimally reordering the facilities within it via an exact MIP model.
In a computational study on benchmark instances from literature, we demonstrate 
consistent improvements over the current 
state-of-the-art metaheuristic by \citet{tang2022solving}, improving the 
best-known solution value for 17 instances and matching the best-known solution 
value  for the remaining 53 instances. We note that the presented work is an 
extension of the Master thesis \citep{pammer2025matheuristic} of the first author of this work.

The rest of this paper is structured as follows:
Section~\ref{sec: pref_rel_work} provides an overview of the previous and related 
work on the SRFLP, categorizing them into metaheuristic and exact approaches. 
Moreover, we also give a short general overview on matheuristics and discuss their 
relation to our presented window approach.
Section~\ref{sec: MIP_window_approach} gives an overview on the design of our 
matheuristic and introduces the the 
MIP-based window approach, while section~\ref{sec: algorithm} describes the 
remaining part of our matheuristic, including implementation details.
Section~\ref{sec: compResults} presents the computational results obtained with 
the window approach matheuristic, comparing them to the current state-of-the-art.
Finally, section~\ref{sec:conclusion} concludes our work and outlines potential 
directions for future research and applications of the presented approach.

\section{Previous and related work}
\label{sec: pref_rel_work}

We first present previous work on the SRFLP, and then give a short overview on 
popular existing matheuristic techniques and discuss their relation to our 
approach.

\subsection{Previous work on the SRFLP}

The SRFLP is one of the most studied problems in facility layout optimization, 
with applications in warehouse design, hospitals, libraries, and manufacturing 
systems \citep{simmons1969one, picard1981one, heragu1988machine}. Due to its 
NP-hardness, researchers have developed both exact and heuristic approaches to 
solve it.

\paragraph{Metaheuristic approaches}

The SRFLP has been extensively studied in the literature, leading to the development of various metaheuristic approaches. These methods can be broadly categorized into construction heuristics, local search methods, and advanced metaheuristics.
Construction heuristics are often the first step in solving the SRFLP, as they provide initial solutions for more complex metaheuristic methods. These heuristics typically focus on the weights and lengths of facilities to generate feasible layouts.
Notable examples include the work of \citet{kumar1995heuristic} and \citet{djellab2001new}, who proposed efficient construction heuristics that prioritize either weights or lengths.

A widely applied metaheuristic for the SRFLP is simulated annealing (SA), which balances exploration and exploitation through a cooling schedule.
Foundational SA-based methods were introduced by \citet{heragu1992experimental} and \citet{romero1990methods}, establishing the approach as a competitive early technique.
Several genetic algorithms have also been proposed for the SRFLP. Notable examples include the work of \citet{ficko2004designing},
who incorporated problem-specific crossover operators, and \citet{datta2011single}, who refined selection mechanisms to improve convergence.
Tabu search (TS) is another class of metaheuristics applied to the SRFLP. \citet{gomes2000metaheuristic} and \citet{kothari2013tabu} developed TS variants tailored to
facility layout constraints, demonstrating improved local search performance.

Bio-inspired algorithms such as ant colony optimization (ACO) and particle swarm optimization (PSO) have also been explored.
ACO was first applied to the SRFLP by \citet{solimanpur2005ant}, while PSO was introduced by \citet{samarghandi2010particle}.
A hybrid method combining both ACO and PSO was presented by \citet{teo2008hybrid}, aiming to leverage the complementary strengths of both algorithms.

Another effective approach is scatter search, which maintains and combines diverse high-quality solutions. \citet{kothari2014scatter} and \citet{satheesh2008scatter} demonstrated its applicability to facility layout problems.

One of the most prominent frameworks in the literature is variable neighborhood search (VNS), which systematically changes neighborhood structures to escape local optima.
The most impactful implementation was developed by \citet{palubeckis2015fast}, while \citet{guan2016hybridizing} extended this approach by integrating ACO into the VNS framework to enhance search intensification.

The most recent and arguably most effective metaheuristic techniques for the SRFLP include the so-called $GRASP_{PR}$, introduced by \citet{rubio2016grasp},
which combines a greedy randomized adaptive search procedure (GRASP) with path relinking for intensification.
An extension of this method,  $GRASP_{F}$, proposed by \citet{cravo2019grasp}, embeds five local search strategies within a VNS-based framework to further refine solutions.
Another competitive approach is the multi-start simulated annealing (MSA) algorithm, introduced by \citet{palubeckis2017single},
which applies SA independently to multiple randomly generated initial permutations,
thus improving solution diversity and robustness. \citet{palubeckis2017single} also proposed efficient implementations of swap and insertion neighborhood operations,
which we adopt in our matheuristic to generate high-quality starting solutions.
To the best of our knowledge, the current state-of-the-art method for large-scale SRFLP instances (up to 1000 facilities) is the k-medoids memetic permutation group ($KMPG$)
algorithm proposed by \citet{tang2022solving}. This algorithm introduces a symmetry-breaking mechanism to reduce the effective search space,
followed by a memetic algorithm that employs k-medoids clustering—an unsupervised learning method—to focus the search on promising solution regions.
An SA-based intensification phase is then applied to locally refine the best solutions found.

\paragraph{Exact methods}

Early exact methods include branch-and-bound \citep{simmons1969one} and dynamic programming \citep{picard1981one}. More recently, cutting-plane and polyhedral techniques have been developed
\citep{anjos2008computing, amaral2013polyhedral}. The two most efficient and most 
commonly used exact approaches to solve the SRFLP are MIP and \emph{semidefinite 
programming (SDP)} methods.
The first mixed-integer formulation, which was still non-linear at the time, was developed by \citet{love1976solving}.
Currently, the most effective approach MIP approach for solving the SRFLP to optimality is the betweenness-based formulation, introduced by \citet{amaral2009new}.
This formulation allows to solve instances with up to 35 facilities to optimality 
and due to 
the availability of efficient MIP solvers such as Gurobi or CPLEX, serves as the 
basis for our window approach which solves parts of the layout to optimality. We 
present this formulation in detail in section \ref{subsec:BBA}.
We note that SDP approaches (see, e.g.~\citep{anjos2005semidefinite, 
anjos2006computational, fischer2006computational}) allow to solve 
instances with up to 81 facilities to 
optimality, as demonstrated by 
\citet{schwiddessen2021semidefinite}. However, these approaches need highly-tuned 
problem-specific solution algorithms based on lagrangian relaxation and have very 
high runtimes. This limits 
their applicability and thus we use the betweenness-based MIP formulation in our 
matheuristic. Finally, note that all existing formulations (MIP and SDP) use at 
least $O(n^2)$ many variables, which limits the scalability of exact approaches, 
keeping in mind that instances tackled by heuristics have up to $n=1000$ 
facilities.

\subsection{Matheuristics}
The general idea of matheuristics is not new to computational optimization and 
there already exist several matheuristic approaches which try to mimic 
classical local search (for a general overview on matheuristics, see, e.g., the 
book \citep{maniezzo2021matheuristics}).

In our proposed window approach, we start with a feasible solution and then 
iteratively apply a (moving) window over this solution, and solve the resulting 
subinstance within the window (i.e., a subset of facilities) to optimality, while 
keeping the facilities outside the window fixed. This gives an improved solution, 
as long as the facilities within the window were not already placed optimal with 
respect to the remaining facilities. 

Thus our matheuristic can be described as "using an exact method to search for 
improved solutions near a given feasible solution". There are two classes of 
matheuristics, which are operating within this paradigm, namely \emph{local 
branching} \citep{fischetti2003local} and the \emph{corridor method} 
\citep{sniedovich2006corridor}. Local 
branching  works by adding local branching constraints to the linear 
programming (LP)-relaxation of the MIP of the considered problem. These local 
branching constraints aim to keep the LP-relaxation near to an already known good 
feasible 
solution, with the hope that the MIP-solver quickly finds an improved solution 
this way. The process is then iterated with the improved solutions as starting 
point. The constraints are typically Hamming-distance type constraints based on 
the current solution. This in particular means that no variables get fixed, and we 
are stuck with solving the original problem (with its original number of 
variables) with additional constraints. In our setting, an approach just adding a 
local branching-type constraint will not be computationally feasible, as for 
larger instances, even solving the LP-relaxation becomes prohibitive due to its 
size.

The idea of the corridor method is quite similar to local branching, quoting from 
the textbook 
of
\citet{maniezzo2021matheuristics}, the idea is "using the exact method over 
successive 
restricted portions of the 
solution space of the given problem. The restriction is obtained by applying 
exogenous constraints, which define local neighborhoods around points of 
interest". When the corridor method was first introduced, the exact method used 
was dynamic programming, but it was observed that it can also work with other 
exact algorithms, such as MIP. Thus, due to the very general definition of what 
constitutes the 
corridor method, our proposed approach could be classified as a type of 
corridor method. However, we note that the exogenous constraint in the 
demonstration of the corridor method in the contexts of MIPs in 
\citet{maniezzo2021matheuristics} is essentially a local branching 
constraint.

Finally, our method can also be viewed as a type of fix-and-relax method.
A well known example of such as type of matheuristic is the \emph{kernel search} 
approach, which was first 
introduced by \citet{angelelli2010kernel} and has been successfully applied to 
various combinatorial optimization problems. It starts with an initial kernel
— a promising subset of decision variables, often selected via LP relaxation — and partitions the remaining variables into buckets.
In each iteration, one bucket is combined with the kernel to form a restricted problem, which is solved under limited computational resources.
If the resulting solution improves the current best, the active variables from that bucket are added to the kernel, which grows monotonically.

\section{The window approach matheuristic}
\label{sec: MIP_window_approach}

Our window approach matheuristic consists of two phases:
\begin{enumerate}
    \item An \textbf{enhanced MSA phase} that generates a high-quality initial solution:
    The core components of our MSA implementation are based on the approach proposed 
    by \citet{palubeckis2017single}.
    In contrast to the work of \citet{palubeckis2017single}, our MSA integrates the \texttt{LS\_wind} algorithm (see algorithm ~\ref{alg:LSWind}), which includes two local search 
    procedures and the moving window approach. The window approach within the MSA 
    is used with a smaller window-size of 13 to ensure a quick solution time.    
    \item A \textbf{refinement phase} that intensifies the best solution found by the first phase again using the \texttt{LS\_wind} algorithm~\ref{alg:LSWind}:
    In this phase, we apply the \texttt{LS\_wind} algorithm with larger window sizes of 17 and 19 to further improve the solution quality.
\end{enumerate}

In the remainder of this section, we focus on the technical details of the window 
approach, i.e., we present the MIP formulation we base our approach on and then 
describe how to modify this MIP to be suitable to use within our window approach. 
In section \ref{sec: algorithm}, we then give implementation details of our 
matheuristic. For the first phase, we focus on the enhancements compared to the 
MSA by 
\citet{palubeckis2017single} (implementation details of the MSA itself are given 
in~\ref{subsec:MSA} with the key enhancements outlined in 
algorithm~\ref{alg:MSA} there).

\subsection{Betweenness-based MIP formulation} \label{subsec:BBA}
The basic idea of the betweenness-based MIP formulation is that the distance between the centers of two facilities can be determined by summing the distances of all facilities located between them,
plus half the length of each of the two facilities. Consequently, to compute the 
distance, it is sufficient to know which facilities are positioned between any 
given pair. Let binary variables $x_{i,k,j}$ be one if and only if facility $k$ is between 
facilities $i$ and $j$ in the solution.
The MIP formulation for the betweenness-based formulation is as follows 
\citep{amaral2009new}.
\begin{align}
    &\textbf{Objective:} \notag \\
    &\min \sum_{i=1}^{n-1} \sum_{j=i+1}^{n} w_{ij} \frac{\ell_i + \ell_j}{2}
    + \sum_{i=1}^{n-1} \sum_{j=i+1}^{n} w_{ij} \sum_{\substack{k \neq i \\ k \neq 
    j}} \ell_k x_{i,k,j} \label{eq:srflp_obj} \\
    &\textbf{Subject to:} \notag \\
    &x_{i,k,j} + x_{i,j,k} + x_{j,i,k} = 1 
    &&\text{for } 1 \leq i < j < k \leq n \label{eq:srflp_c1} \\
    &x_{i,h,j} + x_{i,h,k} + x_{j,h,k} \leq 2 
    &&\text{for } 1 \leq i < j < k < h \leq n \label{eq:srflp_c2} \\
    &x_{i,h,j} + x_{i,h,k} - x_{j,h,k} \geq 0 
    &&\text{for } 1 \leq i < j < k \leq n,\, h \neq i,j,k \label{eq:srflp_c3} \\
    &x_{i,h,j} - x_{i,h,k} + x_{j,h,k} \geq 0 
    &&\text{for } 1 \leq i < j < k \leq n,\, h \neq i,j,k \label{eq:srflp_c4} \\
    &-x_{i,h,j} + x_{i,h,k} + x_{j,h,k} \geq 0 
    &&\text{for } 1 \leq i < j < k \leq n,\, h \neq i,j,k \label{eq:srflp_c5} \\
    &x_{i,k,j} \in \{0,1\}
    &&\text{for } 1 \leq i < j \leq n,\, k \neq i,j \label{eq:srflp_c6}
\end{align}

where the objective (\ref{eq:srflp_obj}) minimzes the weighted sum of distances between all pairs of facilities. Since the first part of the objective function is independent of the permutation,
we can compute it in advance and only need to optimize the second part. 
Constraint~\eqref{eq:srflp_c1} ensures that for any three facilities, exactly one of them is positioned between the other two.
Constraint~\eqref{eq:srflp_c2} ensures that a facility can be positioned between at most two of the other three facilities.
Constraints~\eqref{eq:srflp_c3}-\eqref{eq:srflp_c5} ensure that a facility cannot be positioned between just one pair of the three other facilities.
Figure~\eqref{Fig:constraint4facilities} visualizes the four possible positions of 
the fourth facility $h$, denoted as $h1$ to $h4$. 
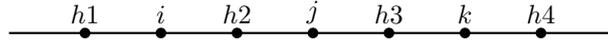
\begin{figure}[!ht]
	\centering
	\begin{tikzpicture}
		\coordinate (i) at (2, 0);
		\coordinate (j) at (4, 0);
		\coordinate (k) at (6, 0);
		\coordinate (h1) at (1, 0);
		\coordinate (h2) at (3, 0);
		\coordinate (h3) at (5, 0);
		\coordinate (h4) at (7, 0);
		
		\fill (i) circle [radius=2pt] node [above] {$i$};
		\fill (j) circle [radius=2pt] node [above] {$j$};
		\fill (k) circle [radius=2pt] node [above] {$k$};
		\fill (h1) circle [radius=2pt] node [above] {$h1$};
		\fill (h2) circle [radius=2pt] node [above] {$h2$};
		\fill (h3) circle [radius=2pt] node [above] {$h3$};
		\fill (h4) circle [radius=2pt] node [above] {$h4$};
		
		\draw[thick] (0, 0) -- (8 , 0);
	\end{tikzpicture}
	\caption{Possible position of fourth facility h}
	\label{Fig:constraint4facilities}
\end{figure}

This shows that the facility can not be positioned
between just one pair of the three other facilities.
The last constraint~\ref{eq:srflp_c6} ensures that the betweenness decision variables are binary.
In the proposed matheuristic, this formulation is extended via the window approach 
to optimally solve subsegments of the SRFLP, while preserving the global structure 
of the solution. Note that directly fixing some variables in the model would not 
translate into a fixing of facilities to positions, as the variables indicate 
betweenness.


\subsection{Window approach}\label{subsec:window}
To apply the betweenness-based MIP formulation to subproblems of the SRFLP for larger 
instances, we developed the so called window approach.
This method leverages the idea of dividing the layout into separate parts. Instead 
of optimizing the entire layout at once, we fix the positions of all facilities 
except for those within a designated window, which is given by the interval 
$[sw,ew]$ (i.e., all facilities with a position within this interval in the given 
solution are in the window).
The SRFLP is then solved optimally for the facilities inside the window, 
considering the fixed positions of those outside it. We first describe how the 
objective function can be partitioned given such a window, and then describe how 
this window approached can be used with the betweenness-based MIP formulation.

The key principle of this method is that distances between facilities can be classified into three categories, depending on how they are influenced by the permutation of facilities
inside the window:
\begin{enumerate}
    \item \textbf{Fixed distances:} Distances between facilities outside the window remain unchanged, as they are independent of the permutation within the window.
    To be more precise: The distances between all pairs of facilities outside the window are fixed, as the total distance between them is unaffected by the permutation inside the window.
    This is illustrated in figure \ref{fig:graph_srflp_outer_facs} 
    \item \textbf{Partially dependent distances:} Distances between a facility inside the window and a facility outside the window are partially fixed.
    The most interesting aspect of this approach is that we can partially fix the distances between facilities inside and outside the window.
    In figure \ref{fig:graph_srflp_inner_outer}, the fixed distances, which can be precomputed are shown in black, while the distances that still need to be computed are highlighted in green.
    \item \textbf{Variable distances:} Distances between facilities inside the window must be computed and optimized, as they totally depend on the permutation inside the window.
    The distances between all pairs of facilities inside the window must still be computed, as they are the primary variables we seek to optimize.
    These distances are visualized with the green arrows in figure 
    \ref{fig:graph_srflp_inner_facs}.
\end{enumerate}

Thus, by precomputing the fixed distances and the fixed part of the partially 
dependent distances, we can focus solely on optimizing the arrangement of 
facilities within the window using
the betweenness-based MIP formulation.

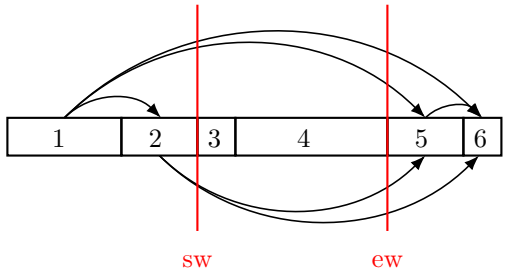
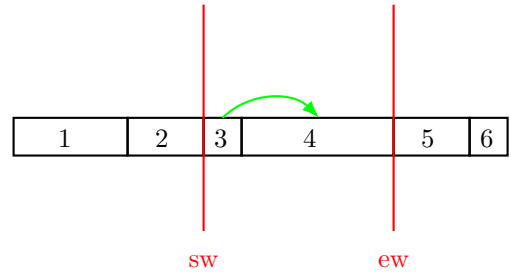
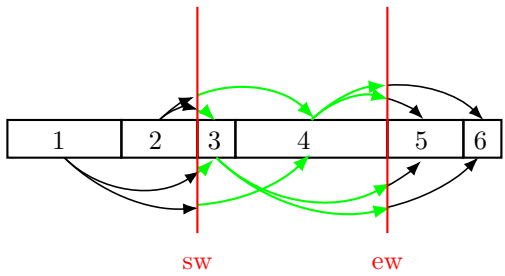
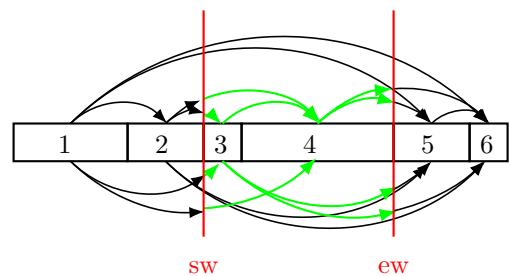
\begin{figure}[h!]
    \begin{minipage}[t]{0.475\textwidth}
    \centering

    \begin{subfigure}[t]{\linewidth}
        \centering
        \begin{tikzpicture}[node distance= 0.2cm, >={Stealth[round]}, thick]

            \draw[draw=black] (0.5,0) rectangle ++(1.5, 0.5) node[pos=.45] {1};
            \draw[draw=black] (2,0) rectangle ++(1, 0.5) node[pos=.45] {2};
            \draw[draw=black] (3,0) rectangle ++(0.5, 0.5) node[pos=.45] {3};
            \draw[draw=black] (3.5,0) rectangle ++(2, 0.5) node[pos=.45] {4};
            \draw[draw=black] (5.5,0) rectangle ++(1, 0.5) node[pos=.45] {5};
            \draw[draw=black] (6.5,0) rectangle ++(0.5, 0.5) node[pos=.45] {6};

            \draw[-Latex, semithick] (1.25, 0.5) arc[radius= 0.9, start angle= 135, end angle=45];
            \draw[-Latex, semithick] (1.25, 0.5) arc[radius= 3.36, start angle= 135, end angle=45];
            \draw[-Latex, semithick] (1.25, 0.5) arc[radius= 3.88, start angle= 135, end angle=45];
            \draw[-Latex, semithick] (2.5, 0) arc[radius= 2.475, start angle= 225, end angle=315];
            \draw[-Latex, semithick] (2.5, 0) arc[radius= 2.975, start angle= 225, end angle=315];
            \draw[-Latex, semithick] (6, 0.5) arc[radius= 0.53, start angle= 135, end angle=45];

            \draw[red, thick] (3, -1) -- (3, 2) node[below] at (3, -1.2) {sw};
            \draw[red, thick] (5.5, -1) -- (5.5, 2) node[below] at (5.5, -1.2) {ew};

        \end{tikzpicture}
        \caption{Distances independent of permutation inside the window}
        \label{fig:graph_srflp_outer_facs}
    \end{subfigure}

    \vspace{0.5cm}

    \begin{subfigure}[t]{\linewidth}
        \centering
        \begin{tikzpicture}[node distance= 0.2cm, >={Stealth[round]}, thick]

            \draw[draw=black] (0.5,0) rectangle ++(1.5, 0.5) node[pos=.45] {1};
            \draw[draw=black] (2,0) rectangle ++(1, 0.5) node[pos=.45] {2};
            \draw[draw=black] (3,0) rectangle ++(0.5, 0.5) node[pos=.45] {3};
            \draw[draw=black] (3.5,0) rectangle ++(2, 0.5) node[pos=.45] {4};
            \draw[draw=black] (5.5,0) rectangle ++(1, 0.5) node[pos=.45] {5};
            \draw[draw=black] (6.5,0) rectangle ++(0.5, 0.5) node[pos=.45] {6};
            \draw[-Latex, semithick, black] (1.25, 0) arc[radius= 1.4142, start angle= 225, end angle=303.75];
            \draw[-Latex, thick, green] (3, -0.2) arc[radius= 1.4142, start angle=303.75, end angle=315];
            \draw[-Latex, semithick, black] (1.25, 0) arc[radius= 2.2981, start angle= 225, end angle=273.46];
            \draw[-Latex, thick, green] (3, -0.625) arc[radius= 2.2981, start angle=273.46, end angle=315];
            \draw[-Latex, semithick, black] (2.5, 0.5) arc[radius= 0.53033, start angle= 135, end angle=75];
            \draw[-Latex, thick, green] (3, 0.625) arc[radius= 0.53033, start angle=75, end angle=45];
            \draw[-Latex, semithick, black] (2.5, 0.5) arc[radius= 1.4142, start angle= 135, end angle=112.5];
            \draw[-Latex, thick, green] (3, 0.825) arc[radius= 1.4142, start angle=112.5, end angle=45];
            \draw[-Latex, thick, green] (4.5, 0.5) arc[radius= 1.06065, start angle= 135, end angle=75];
            \draw[-Latex, semithick, black] (5.5, 0.785) arc[radius= 1.06065, start angle=75, end angle=45];
            \draw[-Latex, thick, green] (4.5, 0.5) arc[radius= 1.591, start angle= 135, end angle=95];
            \draw[-Latex, semithick, black] (5.5, 0.96) arc[radius= 1.591, start angle=95, end angle=45];
            \draw[-Latex, thick, green] (3.25, 0) arc[radius= 1.9445, start angle= 225, end angle=298.6364];
            \draw[-Latex, semithick, black] (5.5, -0.375) arc[radius= 1.9445, start angle=298.6364, end angle=315];
            \draw[-Latex, thick, green] (3.25, 0) arc[radius= 2.475, start angle= 225, end angle=282.8571];
            \draw[-Latex, semithick, black] (5.5, -0.6575) arc[radius= 2.475, start angle=282.8571, end angle=315];

            \draw[red, thick] (3, -1) -- (3, 2) node[below] at (3, -1.2) {sw};
            \draw[red, thick] (5.5, -1) -- (5.5, 2) node[below] at (5.5, -1.2) {ew};

        \end{tikzpicture}
        \caption{Distances partially dependent on permutation inside the window}
        \label{fig:graph_srflp_inner_outer}
    \end{subfigure}
    \end{minipage}
    \hfill
    \begin{minipage}[t]{0.475\textwidth}

    \begin{subfigure}[t]{\linewidth}
        \centering
        \centering
        \begin{tikzpicture}[node distance= 0.2cm, >={Stealth[round]}, thick]

            \draw[draw=black] (0.5,0) rectangle ++(1.5, 0.5) node[pos=.45] {1};
            \draw[draw=black] (2,0) rectangle ++(1, 0.5) node[pos=.45] {2};
            \draw[draw=black] (3,0) rectangle ++(0.5, 0.5) node[pos=.45] {3};
            \draw[draw=black] (3.5,0) rectangle ++(2, 0.5) node[pos=.45] {4};
            \draw[draw=black] (5.5,0) rectangle ++(1, 0.5) node[pos=.45] {5};
            \draw[draw=black] (6.5,0) rectangle ++(0.5, 0.5) node[pos=.45] {6};

            \draw[-Latex, thick, green] (3.25, 0.5) arc[radius= 0.9, start angle= 135, end angle=45];

            \draw[red, thick] (3, -1) -- (3, 2) node[below] at (3, -1.2) {sw};
            \draw[red, thick] (5.5, -1) -- (5.5, 2) node[below] at (5.5, -1.2) {ew};

        \end{tikzpicture}
        \caption{Distances dependent on permutation inside the window}
        \label{fig:graph_srflp_inner_facs}
    \end{subfigure}
    \vspace{0.5cm}

    \begin{subfigure}[t]{\linewidth}
        \centering
        \begin{tikzpicture}[node distance= 0.2cm, >={Stealth[round]}, thick]

            \draw[draw=black] (0.5,0) rectangle ++(1.5, 0.5) node[pos=.45] {1};
            \draw[draw=black] (2,0) rectangle ++(1, 0.5) node[pos=.45] {2};
            \draw[draw=black] (3,0) rectangle ++(0.5, 0.5) node[pos=.45] {3};
            \draw[draw=black] (3.5,0) rectangle ++(2, 0.5) node[pos=.45] {4};
            \draw[draw=black] (5.5,0) rectangle ++(1, 0.5) node[pos=.45] {5};
            \draw[draw=black] (6.5,0) rectangle ++(0.5, 0.5) node[pos=.45] {6};

            \draw[-Latex, semithick] (1.25, 0.5) arc[radius= 0.9, start angle= 135, end angle=45];
            \draw[-Latex, semithick] (1.25, 0.5) arc[radius= 3.36, start angle= 135, end angle=45];
            \draw[-Latex, semithick] (1.25, 0.5) arc[radius= 3.88, start angle= 135, end angle=45];
            \draw[-Latex, semithick] (2.5, 0) arc[radius= 2.475, start angle= 225, end angle=315];
            \draw[-Latex, semithick] (2.5, 0) arc[radius= 2.975, start angle= 225, end angle=315];
            \draw[-Latex, semithick] (6, 0.5) arc[radius= 0.53, start angle= 135, end angle=45];

            \draw[-Latex, thick, green] (3.25, 0.5) arc[radius= 0.9, start angle= 135, end angle=45];

            \draw[-Latex, semithick, black] (1.25, 0) arc[radius= 1.4142, start angle= 225, end angle=303.75];
            \draw[-Latex, thick, green] (3, -0.2) arc[radius= 1.4142, start angle=303.75, end angle=315];
            \draw[-Latex, semithick, black] (1.25, 0) arc[radius= 2.2981, start angle= 225, end angle=273.46];
            \draw[-Latex, thick, green] (3, -0.625) arc[radius= 2.2981, start angle=273.46, end angle=315];
            \draw[-Latex, semithick, black] (2.5, 0.5) arc[radius= 0.53033, start angle= 135, end angle=75];
            \draw[-Latex, thick, green] (3, 0.625) arc[radius= 0.53033, start angle=75, end angle=45];
            \draw[-Latex, semithick, black] (2.5, 0.5) arc[radius= 1.4142, start angle= 135, end angle=112.5];
            \draw[-Latex, thick, green] (3, 0.825) arc[radius= 1.4142, start angle=112.5, end angle=45];
            \draw[-Latex, thick, green] (4.5, 0.5) arc[radius= 1.06065, start angle= 135, end angle=75];
            \draw[-Latex, semithick, black] (5.5, 0.785) arc[radius= 1.06065, start angle=75, end angle=45];
            \draw[-Latex, thick, green] (4.5, 0.5) arc[radius= 1.591, start angle= 135, end angle=95];
            \draw[-Latex, semithick, black] (5.5, 0.96) arc[radius= 1.591, start angle=95, end angle=45];
            \draw[-Latex, thick, green] (3.25, 0) arc[radius= 1.9445, start angle= 225, end angle=298.6364];
            \draw[-Latex, semithick, black] (5.5, -0.375) arc[radius= 1.9445, start angle=298.6364, end angle=315];
            \draw[-Latex, thick, green] (3.25, 0) arc[radius= 2.475, start angle= 225, end angle=282.8571];
            \draw[-Latex, semithick, black] (5.5, -0.6575) arc[radius= 2.475, start angle=282.8571, end angle=315];

            \draw[red, thick] (3, -1) -- (3, 2) node[below] at (3, -1.2) {sw};
            \draw[red, thick] (5.5, -1) -- (5.5, 2) node[below] at (5.5, -1.2) {ew};

        \end{tikzpicture}
        \caption{Putting it all together}
        \label{fig:graph_srflp_comb}
    \end{subfigure}
    \end{minipage}

    \caption{Types of distances in the computation of the objective function given 
    a window defined by $[sw,ew]$.}
    \label{fig:graph_srflp_all}
\end{figure}
In the following proposition we formalize that the objective function of the SRFLP 
can be decomposed in a way that facilitates the application of the window approach.
Specifically, we show that the objective function can be separated into components 
dependent on the permutation within a given window and components that are 
independent
and can thus be precomputed.
\begin{proposition}
    Given the set of facilities inside a window [$sw$,$ew$] in the current solution $\pi$, where $ew > sw$, the objective function of the SRFLP, assuming that the positions of all facilities
    before and after the window are fixed to their position in the current solution $\pi$, can be rewritten as:

\begin{equation}
    \begin{split}
        \text{fixed(sw, ew, }\pi) + \sum_{i=1}^{sw-1} \sum_{j=sw}^{ew} w_{\pi_j, \pi_i} \left( \frac{\ell_{\pi_j}}{2} + \sum_{k=sw}^{j-1} \ell_{\pi_k} \right) + \\
        \sum_{i=sw}^{ew-1} \sum_{j=i+1}^{ew} w_{\pi_i, \pi_j} \left( \frac{\ell_{\pi_i} + \ell_{\pi_j}}{2} + \sum_{k=i+1}^{j-1} \ell_{\pi_k} \right) + \\
        \sum_{i=sw}^{ew} \sum_{j=ew+1}^{n} w_{\pi_i, \pi_j} \left( \frac{\ell_{\pi_i}}{2} + \sum_{k=i+1}^{ew} \ell_{\pi_k} \right) = \\
        \sum_{i=1}^{n-1} \sum_{j=i+1}^{n} w_{\pi_i, \pi_j} \left( \frac{\ell_{\pi_i} + \ell_{\pi_j}}{2} + \sum_{k=i+1}^{j-1} \ell_{\pi_k} \right),
    \end{split}
    \label{eq:obj_to_proof}
\end{equation}

where $fixed(sw, ew, \pi)$ is equal to:
\begin{equation}
    \begin{split}
        \sum_{i=1}^{sw-2} \sum_{j=i+1}^{sw-1} w_{\pi_i, \pi_j} \left( \frac{\ell_{\pi_i} + \ell_{\pi_j}}{2} + \sum_{k=i+1}^{j-1} \ell_{\pi_k} \right) + \\
        \sum_{i=1}^{sw-1} \sum_{j=sw}^{ew} w_{\pi_i, \pi_j} \left( \frac{\ell_{\pi_i}}{2} + \sum_{k=i+1}^{sw-1} \ell_{\pi_k} \right) + \\
        \sum_{i=1}^{sw-1} \sum_{j=ew+1}^{n} w_{\pi_i, \pi_j} \left( \frac{\ell_{\pi_i} + \ell_{\pi_j}}{2} + \sum_{k=i+1}^{sw-1} \ell_{\pi_k} + \ell_w + \sum_{k=ew+1}^{j-1} \ell_{\pi_k} \right) + \\
        \sum_{i=sw}^{ew} \sum_{j=ew+1}^{n} w_{\pi_j, \pi_i} \left( \frac{\ell_{\pi_j}}{2} + \sum_{k=ew+1}^{j-1} \ell_{\pi_k} \right) + \\
        \sum_{i=ew+1}^{n-1} \sum_{j=i+1}^{n} w_{\pi_i, \pi_j} \left( \frac{\ell_{\pi_i} + \ell_{\pi_j}}{2} + \sum_{k=i+1}^{j-1} \ell_{\pi_k} \right),
    \end{split}
    \label{eq:fixed}
\end{equation}

as these distances do not depend on the permutation within the window, they can be computed in advance.
Consequently, when optimizing the permutation between two positions with respect to the fixed layout outside the window,
the objective function can be optimized by considering only the last three summands.
\end{proposition}

\begin{proof}
See~\ref{app:proof}.
\end{proof}


Thus, for any facility located before the window, we can fix its distance to the 
start of the window, based on its center position, as illustrated in 
figure~\ref{fig:graph_srflp_inner_outer}.
Analogously, for a facility positioned after the window, we fix the distance from the end of the window to the center of that facility.

To apply the betweenness-based MIP formulation to a given solution, we proceed in two steps:

\begin{enumerate}
\item \textbf{Introduction of two dummy facilities:} We introduce two dummy facilities that represent the aggregated facilities located before and after the window.
For the MIP model, one dummy facility is introduced for the set of facilities preceding the window and one for those following it. Both dummy facilities are assigned a length of zero and fixed
at predefined positions, ensuring that only the relative ordering of the facilities within the window is subject to optimization.
To enforce this structure, we include the following constraint:

\begin{equation}
\sum_{k=1}^{ws} x_{ws+1,k,ws+2} = ws,
\label{eq:dummyconstraint}
\end{equation}

where, $ws + 1$ denotes the dummy facility representing all facilities preceding the window, while $ws + 2$ represents those following it.
This constraint ensures that all facilities within the window are placed between the two dummy nodes, thereby preserving the fixed external structure of the layout.
\item \textbf{Modification of weights matrix:} To account for the weigths between internal and external facilities, we aggregate the weights
between each facility inside the window and the facilities located before and after it. Specifically, for every facility within the window,
we compute the cumulative weights from all external facilities and add them to the corresponding entries for the dummy facility in the weight matrix $w_{ij}$.
This adjustment ensures that the contribution of the weigths of the facilities outside the window is accurately represented during the localized optimization. 
This way we can use the general objective function for the MIP after this modification, see equation \eqref{eq:final_obj}.
\end{enumerate}

We can then solve the problem regarding the permutation of the facilities inside 
the window to optimality with respect to the fixed positions outside the window.
This method significantly reduces computation time and by keeping the window size 
small (i.e., between 13 and 19 as discussed in the beginning of the section), we 
can ensure that the problem remains computationally manageable.

\section{Algorithmic design}
\label{sec: algorithm}
In this section, we present the detailed design of our matheuristic algorithm, 
which integrates the components introduced in section~\ref{sec: 
MIP_window_approach}.
We first 
outline the implementation of the metaheuristic components, and 
then describe the integration and execution of the window approach.

\subsection{Implementation of the metaheuristic components}
\label{subsec:metaheuristic}
Our approach starts with an MSA phase, closely following \citet{palubeckis2017single}, which restarts a SA procedure multiple times.
Palubeckis introduced a novel approach to efficiently compute the move gains for the SRFLP using specific vectors and matrices.
For each temperature level $T$, $n \cdot \hat{z}$, where $\hat{z}$ is the number of inner runs, random moves are computed: a swap move with probability $p$
and an insertion move with probability $1-p$. Depending on the temperature level, two distinct computational methods are utilized for generating and applying move gains.
Specifically, one method is more efficient when the probability of accepting a move is relatively high; hence, it is employed at higher temperatures.
Conversely, the alternative method excels in faster computation of move gains and is thus preferred at lower temperatures where the acceptance of a move is already relatively low. 
The selected move is applied depending on the move gain and the current temperature.
Afterwards, the temperature is reduced by a cooling factor of $\alpha$. The parameters of the MSA are tested in subsection \ref{subsec: paramChoice}.

For ease of readability, we focus on the differences 
to \citet{palubeckis2017single} in the remaining part of this section, more 
details about the MSA can be found in~\ref{subsec:MSA}.
To intensify solutions obtained from the MSA, we apply a procedure combining local 
search techniques with the window approach in an iterative fashion. This hybrid 
phase is shown in algorithm~\ref{alg:LSWind}.
Initially, we apply the \texttt{LS\_insert} local search, as insertion moves have 
been shown to be more effective than swap moves in similar contexts 
\citep{cravo2019grasp}. 
Before invoking this phase, we precompute the required matrices 
and vectors, as shown in the MSA setup by \citet{palubeckis2017single} to ensure time efficient computation of all move gains.
However in contrast to \citet{palubeckis2017single}, we do not only use these matrices for the MSA phase but also adapted his approach to use it for the local search phase.

Following \texttt{LS\_insert}, we execute the \texttt{LS\_swap} procedure iteratively until no further improvements are achieved. Afterwards, the window approach method \texttt{wind\_met} is applied.
If \texttt{wind\_met} yields an improved solution, the matrices and vectors are recomputed to reflect the updated configuration, and the local search procedures are repeated based on the new solution.
Window sizes are set adaptively: within the MSA, we use the window size vector 
$wsv = (13)$; outside the 
MSA, i.e., in the refinement phase, we use $wsv = (17, 19)$ to allow for 
potentially more 
improvement (of course, using larger windows results in more time consuming MIPs). 
The "go-back" parameter $gb$ controls, if we directly leave \texttt{wind\_met} 
after the first improvement (which we do during the MSA phase, where we are 
interested in quick runtime), or if we continue with the remaining windows to 
potentially further improve the solution (which we do in the refinement phase). 

\begin{algorithm}[!ht]
\caption{\texttt{LS\_wind}}
\label{alg:LSWind}
\KwIn{$\pi$, $F(\pi)$, $wsv$, $gb$}
\KwOut{$\pi$, $F(\pi)$}
\While{$impr$}{
    \While{$loc\_impr$}{
        $\pi, F(\pi) \gets$ \texttt{LS\_insert}$(\pi, F(\pi))$\;
        $\pi, F(\pi), loc\_impr \gets$ \texttt{LS\_swap}$(\pi, F(\pi))$\;
    }
    $\pi, F(\pi), \texttt{impr} \gets$ \texttt{wind\_met}$(\pi, F(\pi), wsv, gb)$\;
}
\Return{$\pi$, $F(\pi)$}\;
\end{algorithm}

\paragraph{Implementation of the local search methods}
In contrast to the MSA approach, where random move gains are computed, we now 
compute the move gains for all possible insertions in the current solution $\pi$, 
using the move computation method which excels in faster computation of move gains as proposed by \citet{palubeckis2017single}.
This method is optimized for efficient move gain computation, although it requires more time to update the solution, as move gains for all insertion pairs must be evaluated,
even though at most one move is performed.
As shown in algorithm~\ref{alg:LSInsert}, we always perform the move with the best gain and recompute the move gains for the updated solutions, iterating until no further improvement is possible.
Detailed implementation details are provided in \citet{pammer2025matheuristic}.

\begin{algorithm}[h!t]
    \caption{\texttt{LS\_insert}}
    \label{alg:LSInsert}
    \KwIn{$\pi$, $F(\pi)$}
    \KwOut{$\pi$, $F(\pi)$}

    $gm \gets$ compute move gains for all combination of facilities and positions\;
    \While{$\min(gm) < 0$}{
        $k,l \gets \argmin gm$\;
        $r \gets \pi_k$\;
        $F(\pi) \gets F(\pi) + \min(gm)$\;
    	insert facility $r$ on position $l$ in the current solution $\pi$\ \;
        $gm \gets$ recompute move gains for all combination of facilities and positions\;
    }
    \Return{$\pi$, $F(\pi)$}\;
\end{algorithm}

Similarly, the \texttt{LS\_swap} procedure, outlined in algorithm~\ref{alg:LSSwap}, evaluates all pairwise facility swaps.
In contrast to the \texttt{LS\_insert}, the \texttt{LS\_swap} algorithm immediately returns the updated solution $\pi$, its objective value $F(\pi)$ and a boolean variable $impr$,
as the local search operators are used in iterative fashion.

\begin{algorithm}[!ht]
\caption{\texttt{LS\_swap}}
\label{alg:LSSwap}
\KwIn{$\pi$, $F(\pi)$}
\KwOut{$\pi$, $F(\pi)$, $impr$}

$impr \gets$ false\;
$gm \gets$ compute move gains for all pairs of facilities\;

\If{$\min(gm) < 0$}{
    $impr \gets$ true\;
    $r, s \gets \argmin(gm)$\;
    $F(\pi) \gets F(\pi) + \min(gm)$\;
    swap positions of facility $r$ and $s$ in the current solution $\pi$\;
}
\Return{$\pi$, $F(\pi)$, $impr$}\;
\end{algorithm}

\subsection{Implemention of the window approach}

For every window, which first consists of the first $ws$ facilities, in the current solution $\pi$ the \texttt{wind\_met} method, shown in algorithm~\ref{alg:windMet}, 
optimizes the permutation of the facilities inside the window with respect to the fixed positions of the facilities outside the window.
It starts by computing the summed up weights of the facilities before and after the window to each of the facilities inside the window and adds them to the weight matrix of the facilities
inside the window as described in subsection~\ref{subsec:window}.
The method then computes the fixed weighted distances outside the window, visualized with black arrows in figure~\ref{fig:graph_srflp_comb}, using the algorithm \texttt{calc\_out\_w\_dist} (see algorithm ~\ref{alg:outwdist} in \ref{app:appendixB} for details).
Afterwards, the adapted betweenness-based MIP formulation, described in subsection~\ref{subsec:window}, is invoked to solve the permutation of the facilities inside the window to optimality.
If the MIP model finds an improved solution, the current solution $\pi$ is updated using the function \texttt{create\_order} shown in algorithm~\ref{alg:order}.
Depending on $gb$, which is true outside the \texttt{MSA} and false inside it, the method either returns the updated solution and its objective value immediately or continues to the next window which is 
shifted by $ws$ facilties. After the method has iterated through all windows of the current window size, it checks whether any improvements were found.
If so, it returns the updated solution and its objective value. Otherwise, it increments the window size and restarts the process with the new window size.

\begin{algorithm}[!ht]
	\caption{\texttt{wind\_met}}
	\label{alg:windMet}
	\KwIn{$\pi$, $F(\pi)$, $wsv$, $gb$}
	\KwOut{$\pi$, $F(\pi)$, $loc\_impr$}
	
	$\pi^*$ $\gets$ $\pi$\;
	$h \gets 1$\;
	$loc\_impr \gets$ false\;
	\While{$h \leq len(wsv)$}{
		$ws \gets wsv[h]$\;
		
		\For{$i \gets 1$ \KwTo $n / ws$}{
            $\pi^{'} \gets \pi$\;
			$wf \gets$ facilities inside the window in $\pi^{'}$\;
			$wb, wa \gets$ sum of weights to facilities before/after the window for all facilities in window\;
			$ww \gets$ weights of facilities inside the window, including $wb$ and $wa$\;
			$fwd \gets$ \texttt{calc\_out\_w\_dist}($\pi^{'}, sw, ew$)\;
			$x \gets$ \texttt{MIP\_BB}($wf, ww, ws, fwd, F(\pi^{'})$)\;
			$pos \gets$ \texttt{create\_order}($x$, $ws$)\;
			update $\pi^{'}$ inside the window according to $pos$\;
			\If{$F(\pi^{'})<F(\pi)$}{
				$\pi \gets \pi^{'}$\;				
				\If{$gb$}{
					$loc\_impr \gets$ true\;
					\Return{$\pi$, $F(\pi)$, $loc\_impr$}\;
				}
			}
		}
		
		\If{$F(\pi)$ $<$ $F(\pi^*)$}{
			$loc\_impr \gets$ true\;
			\Return{$\pi$, $F(\pi)$, $loc\_impr$}\;
		}
		\Else{
			$h \gets h + 1$\;
		}
	}
	\Return{$\pi$, $F(\pi)$, $loc\_impr$}\;
\end{algorithm}
\clearpage

To further enhance the performance of the moving window approach we incorporate 
two acceleration strategies:

\begin{enumerate}
    \item \textbf{Fixed minimum cumulative weigths: } The first acceleration strategy fixes parts of the 
    summed up weights of the facilities before and after the window to each of the 
    facilities inside it.
        To be more precise, the underlying idea is that the weight from each facility inside the window to the facilities before and after the window must be at least the minimum of those two weights.
        Therefore, we can add the minimum of the summed up weights to the facilities before and after the window—multiplied by the combined length of all facilities inside the window—directly to the objective function.
        These minimum-weighted distances are independent of the permutation within the window, which justifies the use of the following equation:
        \begin{equation}
            \text{min\_wd} = \sum_{i \in wf} \ell_{i} \cdot  \sum_{i=1}^{ws} min\_weights[i],
            \label{eq:minFixedCosts}
        \end{equation}
        where $min\_weights[i]$ is equal to the minimum summed up weights of the $i$'th facility inside the window to the facilities before and after the window.
        This also means that we have to subtract the minimum of the summed up weights before and after for each of the facilities inside the window in the weights matrix.
        This makes the preprocessing part even faster because there are now $2 \cdot ws$ more zeros in the weights matrix.
    \item \textbf{Improvement constraint: } The second speed up idea ensures that the MIP model only returns a solution if an actual improvement is found.
        This is implemented to avoid unnecessary computational effort and to ensure time efficiency during the optimization. This requirement is enforced using the following constraint:
        \begin{equation}
            fixed(sw,ew,\pi) + \sum_{i=1}^{ws-1}\sum_{j=i+1}^{ws} w_{ij} \sum_{k 
            \neq i, k \neq j} \ell_k x_{i,k,j} < F(\pi),
            \label{eq:improvement}
        \end{equation}
        where $fixed(sw, ew, \pi)$ stands for the total fixed weighted distances 
        for the current window in the current solution $\pi$.
        Here, the left-hand side represents the objective value of the solution, including the fixed contribution from the rest of the permutation. 
        The constraint ensures that this value must be strictly less than the 
        current objective value $F(\pi)$, thereby guaranteeing that only 
        improvements are accepted.
    \end{enumerate}

\section{Computational results}
\label{sec: compResults}

In this section, we compare the performance of our matheuristic approach with the current state-of-the-art metaheuristics for solving the SRFLP.

\subsection{Experimental setup}
\label{sec: expSet}

Our matheuristic approach was implemented in Julia, utilizing Gurobi 12.0.0 as the 
MIP solver to solve the MIP formulation described in section~\ref{sec: 
MIP_window_approach}. 
All experiments were conducted on a single core of an Intel Xeon X5570 processor 
running at 2.93 GHz and 6GB memory. The runtime for each instance was set to 
$n^{1.7}$ seconds, where $n$ denotes the number of facilities in the respective 
problem instance. 
This time limit scales with the problem size to account for the increased 
computational effort required for larger instances (note that other works in 
literature also adapt the runtime of their metaheuristics to the instance size).
To evaluate the performance of our method and facilitate a direct comparison with 
the current state-of-the-art metaheuristics for the SRFLP, we selected three 
widely used benchmark instance sets with large-scale instances.
These benchmark instances have also been employed in recent literature, such as \citet{rubio2016grasp, cravo2019grasp, tang2022solving}.

\begin{itemize}
    \item \texttt{Anjos-large:} This instance set comprises 20 instances. Each instance is defined by a specific number of facilities, a symmetric weight matrix, and a corresponding length vector. The number of facilities ranges from 200 to 500,
    with 5 instances for each of the values 200, 300, 400, and 500. The entries in the weight matrix lie between 0 and 16, while the length vector values range from 0 up to 499.
    \item \texttt{Sko-large:} This set also includes 20 instances with the same 
    facility sizes as the \texttt{Anjos-large} set. The symmetric weight matrix in 
    these instances contains values between 0 and 10,
    and the entries of the length vector range from 0 to 20. 
    \item \texttt{Palubeckis-large:} This set includes 30 instances with 310 to 
    1000 facilities. The symmetric weight matrix in these instances contains 
    values between 0 and 10,
    and the entries of the length vector range from 1 to 10. 
\end{itemize}

To assess the performance of our algorithm, we evaluate not only the best solution obtained but also the average objective value across 15 independent runs. This allows us to provide insights into the robustness and consistency of our method.

\subsection{Choice of parameters} \label{subsec: paramChoice}

Before executing our matheuristic, it is necessary to determine suitable values 
for the parameters used within the MSA phase.
Specifically, we tune three key parameters: the cooling factor ($\alpha$), the 
number of inner runs ($\hat{z}$), and the probability of generating a swap move 
($p$).
To ensure a fair and systematic comparison while avoiding overfitting the algorithm to the final evaluation set, we performed parameter tuning exclusively on a separate training set.
Specifically, we used 20 instances of sizes 150 and 450 from the 
\texttt{Anjos-large} dataset, also used as a trainings data set by 
\citet{rubio2016grasp}. 
We adopted a one-at-a-time tuning strategy: for each parameter under 
consideration, the other two were fixed to values close to those proposed by 
\citet{palubeckis2017single}.
This approach allows us to isolate the impact of each parameter on solution quality, while minimizing interaction effects between parameters during tuning.

\paragraph{Cooling factor}  
We began by tuning the cooling factor $\alpha$, which controls the temperature cool down rate in the SA process. Palubeckis set this parameter to 0.95 in his original work.
We tested three values: 0.95, 0.98, and 0.99, while fixing $\hat{z}$ to 100 and $p$ to 0.35. These fixed values were chosen based on Palubeckis' MSA algorithm.

\begin{table}[ht]
\centering
\begin{tabular}{lr}
  \hline
Cooling factor & mean obj. val \\ 
  \hline
0.95 & 1541565894.94 \\ 
  0.98 & 1531806007.82 \\ 
  0.99 & 1593809235.81 \\ 
   \hline
\end{tabular}
\caption{Cooling factor comparison for different parameter settings.}
\label{tab:cooling-fac}
\end{table}

As shown in table \ref{tab:cooling-fac}, a cooling factor of 0.98 led to slightly better average results.
This suggests that extending the duration of a single SA run tends to be more effective than relying on frequent restarts. Based on this observation, we set $\alpha = 0.98$ for all subsequent experiments.

\paragraph{Number of inner runs}  
Next, we tuned the number of inner runs $\hat{z}$, which determines how many moves are generated per temperature level. This value is multiplied by the number of facilities $n$,
making it sensitive to the instance size. Palubeckis fixed this parameter to 100, and we evaluated the alternatives 70, 100, and 130, while using the $\alpha = 0.98$ and $p = 0.35$.

\begin{table}[ht]
\centering
\begin{tabular}{lr}
  \hline
Numb. inner runs times n & mean obj. val \\ 
  \hline
70 & 1532770744.03 \\ 
  100 & 1531811594.45 \\ 
  130 & 1531832981.13 \\ 
   \hline
\end{tabular}
\caption{Number inner runs comparison for different parameter settings.}
\label{tab:n-inner-runs}

\end{table}

As shown in table~\ref{tab:n-inner-runs}, using $100 \cdot n$ inner moves per 
temperature yielded the best average performance. Consequently, we fixed $\hat{z} 
= 100$ for the remainder of the experiments.

\paragraph{Probability of swap moves}  
Finally, we evaluated the probability $p$ of generating a swap move versus an insertion move. In Palubeckis' study, values of 0, 0.5, and 1 were tested,
with the best results achieved at 0.5 and the second-best at 0. We focused our search around these promising values and tested 0.30, 0.35, and 0.40,
while keeping $\alpha = 0.98$ and $\hat{z} = 100$. Note that the complementary probability determines the likelihood of applying an insertion move instead.

\begin{table}[ht]
\centering
\begin{tabular}{lr}
  \hline
  Prob. to swap & mean obj. val \\ 
  \hline
  0.3 & 1531802502.24 \\ 
  0.35 & 1531800292.13 \\ 
  0.4 & 1531813084.39 \\ 
   \hline
\end{tabular}
\caption{Probability for swap move comparison for different parameter settings.}
\label{tab:prob-swap}
\end{table}

As shown in table~\ref{tab:prob-swap}, a swap probability of 0.35 resulted in the best average performance. Thus, we fixed $p = 0.35$ for all further experiments.

\subsection{Improvement found due window approach}
\label{sec: impr_wind}
Before comparing the performance of our innovative window-based matheuristic with existing state-of-the-art metaheuristics, we first analyze the specific impact of the window approach on the objective values
of the test instances. 
To this end, we evaluate our algorithm both with and without the window component, using the same fixed parameter configuration for each version.
Each variant is executed 15 times on all three benchmark instance sets, ensuring a fair and consistent comparison.
This experimental setup allows us to isolate and quantify the contribution of the window approach in achieving high-quality solutions and in reaching and outperforming current state-of-the-art methods.
The results for the \texttt{Anjos-large} instance set are summarized in table~\ref{tab:anjos_large_diff}. The first column lists the instance name,
where the first part of the integer indicates the instance size. The second column reports the best objective value found using the full algorithm with the window component,
and the third column reports the best value obtained when omitting the window approach.

\begin{table}[ht]
\centering
\begin{tabular}{lrr}
  \hline
instance & window\_appr & wo\_window\_appr \\ 
  \hline
  200\_01 & 305461818.00 & 305461818.00 \\ 
  200\_02 & 178806828.50 & 178806828.50 \\ 
  200\_03 & 61891275.00 & 61891275.00 \\ 
  200\_04 & 127735691.00 & 127735691.00 \\ 
  200\_05 & 89057121.50 & 89057121.50 \\ 
  300\_01 & 1549422266.00 & 1549422266.00 \\ 
  300\_02 & 955538302.50 & 955538302.50 \\ 
  300\_03 & 308257630.50 & 308257630.50 \\ 
  300\_04 & 602873168.50 & 602873168.50 \\ 
  300\_05 & 466150775.00 & 466150775.00 \\ 
  400\_01 & 4999801252.50 & 4999801252.50 \\ 
  400\_02 & 2910265439.00 & 2910265439.00 \\ 
  400\_03 & \textbf{920860291.00} & 920860340.00 \\ 
  400\_04 & 1805638949.00 & 1805638949.00 \\ 
  400\_05 & 1401831869.50 & 1401831869.50 \\ 
  500\_01 & \textbf{12290667266.00} & 12290670220.00 \\ 
  500\_02 & \textbf{7491697208.50} & 7491701877.50 \\ 
  500\_03 & \textbf{2478331774.00} & 2478331869.00 \\ 
  500\_04 & \textbf{4281106062.50} & 4281106067.50 \\ 
  500\_05 & \textbf{3675292982.50} & 3675293102.50 \\ 
   \hline
\end{tabular}
\caption{Comparison of objective function values with and without window approach - \texttt{Anjos-large} instances}
\label{tab:anjos_large_diff}
\end{table}

As seen in table~\ref{tab:anjos_large_diff}, the algorithm using the window approach consistently matches or outperforms the version without the window approach.
Although all runs are entirely independent and both methods operate under identical time limits—which allows the variant without the MIP-based window approach
to explore a larger number of independent starting solutions—the window-enhanced version still achieves superior or equal results across all instances.
A closer examination of the results obtained with both configurations of our matheuristic reveals that, for instance sizes with up to 300 facilities,
the algorithm already finds the best solutions using only the MSA and local search 
components (we note that in the implementation of the MSA 
in~\citep{palubeckis2017single}, there is no local search phase at the end). As 
shown later in subsection~\ref{subsec: res_anjos},
these solutions already represent the current state-of-the-art for these instances.
This demonstrates the robustness of the metaheuristic components employed within the window-based matheuristic.
When focusing on the \texttt{Anjos-large} instances with at least 400 facilities, we observe that the algorithm incorporating the MIP-based window approach outperforms the alternative in six out of ten cases.
Of particular interest is that the window-based version achieves better solutions for all instances with 500 facilities, clearly indicating the effectiveness of the window approach
for solving larger instances.

\begin{table}[ht]
\centering
\begin{tabular}{lrr}
  \hline
instance & window\_appr & wo\_window\_appr \\ 
  \hline
  200\_01 & 3230706.00 & 3230706.00 \\ 
  200\_02 & 7758927.00 & 7758927.00 \\ 
  200\_03 & 12739043.00 & 12739043.00 \\ 
  200\_04 & \textbf{20260531.00} & 20260540.00 \\ 
  200\_05 & 26871976.50 & 26871976.50 \\ 
  300\_01 & \textbf{11249787.00} & 11249806.00 \\ 
  300\_02 & \textbf{28991767.00} & 28991771.00 \\ 
  300\_03 & 48790881.50 & 48790881.50 \\ 
  300\_04 & 71185096.50 & 71185096.50 \\ 
  300\_05 & 86789519.50 & 86789519.50 \\ 
  400\_01 & \textbf{26696000.00} & 26696018.00 \\ 
  400\_02 & \textbf{67950476.00} & 67950478.00 \\ 
  400\_03 & \textbf{115881368.00} & 115881369.00 \\ 
  400\_04 & 162932778.00 & 162932778.00 \\ 
  400\_05 & 227632499.50 & 227632499.50 \\ 
  500\_01 & \textbf{52853051.00} & 52853111.00 \\ 
  500\_02 & \textbf{127362131.00} & 127362134.00 \\ 
  500\_03 & \textbf{230877706.50} & 230877709.50 \\ 
  500\_04 & \textbf{340273141.00} & 340273149.00 \\ 
  500\_05 & \textbf{445698560.00} & 445698566.00 \\ 
   \hline
\end{tabular}
\caption{Comparison of objective function values with and without window approach - \texttt{Sko-large} instances}
\label{tab:sko_large_diff}
\end{table}

We now proceed with the comparison of the two algorithm variants on the 
\texttt{Sko-large} instance set in table~\ref{tab:sko_large_diff}. As seen in the 
table the algorithm incorporating the window approach again consistently matches 
or outperforms the version without the window component.

A more detailed examination of the results reveals that, for the \texttt{Sko-large} instances, the window-based algorithm already produces better solutions in three cases with up to 300 facilities.
This suggests that the \texttt{Sko-large} instances are more challenging to solve using purely metaheuristic methods, whereas the inclusion of the window approach still enables us to achieve solutions
that match the current state-of-the-art.
This observation provides an early indication of the strong robustness of the window-based matheuristic, which will be examined in more depth in the following section.
For instances with at least 400 facilities, the algorithm using the window approach yields better results in eight out of ten cases.
Notably, it again outperforms the non-window version across all instances with 500 facilities, further underscoring the effectiveness of the window strategy in tackling large-scale problems.

We now also compare the differences between the algorithm using the window approach and the one without it on the \texttt{Palubeckis-large} instance set.
\begin{table}[ht]
\centering
\begin{tabular}{lrr}
  \hline
instance & window\_appr & wo\_window\_appr \\ 
  \hline
310 & 105754955.00 & 105754955.00 \\ 
320 & 119522881.50 & 119522881.50 \\ 
330 & 124891823.50 & 124891823.50 \\ 
340 & 129796777.50 & 129796777.50 \\ 
350 & 149594388.00 & 149594388.00 \\ 
360 & 141122187.00 & 141122187.00 \\ 
370 & 174663159.50 & 174663159.50 \\ 
380 & 189452403.50 & 189452403.50 \\ 
390 & 208688557.00 & 208688557.00 \\ 
400 & 213768810.50 & 213768810.50 \\ 
410 & 243494269.00 & 243494269.00 \\ 
420 & 270756527.50 & 270756527.50 \\ 
430 & \textbf{286334521.50} & 286339945.50 \\ 
440 & 301067264.50 & 301067264.50 \\ 
450 & 324488485.00 & 324488485.00 \\ 
460 & 314884659.00 & 314884659.00 \\ 
470 & \textbf{379529990.00} & 379533805.00 \\ 
480 & 366821074.00 & 366821074.00 \\ 
490 & 413901954.50 & 413901954.50 \\ 
500 & 465570835.50 & 465570835.50 \\ 
550 & 587090450.50 & 587090450.50 \\ 
600 & \textbf{801567640.50} & 801567648.50 \\ 
650 & \textbf{927512834.00} & 927512844.00 \\ 
700 & \textbf{1158462242.00} & 1158462275.00 \\ 
750 & \textbf{1438408836.50} & 1438408838.50 \\ 
800 & \textbf{1861593361.00} & 1861593385.00 \\ 
850 & \textbf{2126675899.00} & 2126675902.00 \\ 
900 & \textbf{2600003520.00} & 2600003539.00 \\ 
950 & \textbf{2993153552.50} & 2993153653.50 \\ 
1000 & \textbf{3426646843.00} & 3426649204.00 \\ 
   \hline
\end{tabular}
\caption{Comparison of objective function values with and without window approach - \texttt{Palubeckis-large} instances}
\label{tab:palu_large_diff}
\end{table}

As shown in table~\ref{tab:palu_large_diff}, the window-based algorithm consistently matches or improves upon the results of the non-window version.
For the \texttt{Palubeckis-large} instances, the method using the window approach finds lower upper bounds in 11 out of 30 cases.
Furthermore, as will be discussed later, our window-based method finds better upper bounds precisely in those cases where our matheuristic outperforms the current state-of-the-art approaches.
These results underline the effectiveness of the window approach in systematically exploring previously unexamined neighborhoods, particularly beneficial for larger and more complex instances.
Overall, the comparison emphasizes the strength of combining targeted mathematical optimization via the adapted betweenness-based MIP formulation with local search strategies. This combination provides a powerful approach for enhancing already high-quality solutions.

In summary, the window approach proves to be a crucial component for effective solution intensification.
Its innovative design enables the exploration of previously unexamined neighborhoods, making it particularly powerful for tackling large-scale instances.

\subsection{Comparing window approach matheuristic to current state-of-the-art 
algorithms}

To the best of our knowledge, the $KMPG$ algorithm proposed by~\citet{tang2022solving} represents the most advanced approach
for solving the SRFLP to date. This algorithm has produced the best-known solutions for all instances in all three benchmark sets introduced in section \ref{sec: expSet}.
Notably, $KMPG$ incorporates a k-medoids-based crossover operator — an advanced technique grounded in unsupervised learning — to enhance solution recombination within its memetic framework.
In addition to $KMPG$, we include the most effective classical metaheuristic, $GRASP_F$ \citep{cravo2019grasp}, in our experimental comparison.
These algorithms have previously achieved competitive results on the same 
benchmark instances and serve as well-established reference points for evaluating 
the performance of our method.

\paragraph{Computational results for \texttt{Anjos-large} instance set} \label{subsec: res_anjos}
We begin the comparison with the \texttt{Anjos-large} instance set. As mentioned previously, the $KMPG$ algorithm has achieved the best known solutions for all instances in this benchmark set,
while $GRASP_F$ consistently produced the second-best results on average. Thus, their results serve as reference points for evaluating the performance of our window-based matheuristic.
The results are presented in table~\ref{tab:anjos_comb_results}. The instance size is indicated in the first part of the first column.
The second and third columns show the best known objective values reported by $KMPG$ and $GRASP_F$, respectively.
The last three columns report the best, average, and standard deviation of the objective values achieved by our matheuristic across 15 independent runs. 
Since the algorithms differ in implementation language, hardware, and execution environments, we do not include computational runtime in this comparison.
Instead, we set a uniform runtime limit of $n^{1.7}$ seconds for our matheuristic,
where $n$ denotes the number of facilities in the respective instance. Our focus is on solution quality, measured by the objective value.

\begin{table}[ht]
\centering
\begin{tabular}{lrr|rrr}
  \hline
  & $KMPG$ & $GRASP_F$ & \multicolumn{3}{c}{window\_appr} \\
instance & & & BEST & AVG & SD \\ 
  \hline
200\_01 & \textbf{305461818.0} & \textbf{305461818.0} & \textbf{305461818.0} & 305462878.56 & 1722.21 \\
200\_02 & \textbf{178806828.5} & \textbf{178806828.5} & \textbf{178806828.5} & 178807306.21 & 1754.03 \\
200\_03 & \textbf{61891275.0} & \textbf{61891275.0} & \textbf{61891275.0} & 61891275.00 & 0.00 \\
200\_04 & \textbf{127735691.0} & 127735701.0 & \textbf{127735691.0} & 127735783.89 & 232.62 \\
200\_05 & \textbf{89057121.5} & 89057130.5 & \textbf{89057121.5} & 89057121.50 & 0.00 \\
300\_01 & \textbf{1549422266.0} & 1549423272.0 & \textbf{1549422266.0} & 1549431740.07 & 12434.93 \\
300\_02 & \textbf{955538302.5} & \textbf{955538302.5} & \textbf{955538302.5} & 955538494.58 & 359.17 \\
300\_03 & \textbf{308257630.5} & 308257645.5 & \textbf{308257630.5} & 308257630.50 & 0.00 \\
300\_04 & \textbf{602873168.5} & \textbf{602873168.5} & \textbf{602873168.5} & 602873168.50 & 0.00 \\
300\_05 & \textbf{466150775.0} & 466151295.0 & \textbf{466150775.0} & 466154239.90 & 9439.93 \\
400\_01 & \textbf{4999801252.5} & 4999816802.0 & \textbf{4999801252.5} & 4999812094.57 & 15881.34 \\
400\_02 & \textbf{2910265439.0} & 2910265496.0 & \textbf{2910265439.0} & 2910265448.11 & 38.55 \\
400\_03 & \textbf{920860291.0} & 920861309.0 & \textbf{920860291.0} & 920861837.61 & 7401.97 \\
400\_04 & \textbf{1805638949.0} & \textbf{1805638949.0} & \textbf{1805638949.0} & 1805638949.00 & 0.00 \\
400\_05 & \textbf{1401831869.5} & 1401835315.0 & \textbf{1401831869.5} & 1401831869.50 & 0.00 \\
500\_01 & \textbf{12290667266.0} & 12290696250.0 & \textbf{12290667266.0} & 12290731246.57 & 186431.01 \\
500\_02 & 7491697298.5 & 7491700551.0 & \textbf{7491697208.5} & 7491824038.67 & 272529.95 \\
500\_03 & \textbf{2478331774.0} & 2478348156.0 & \textbf{2478331774.0} & 2478434979.30 & 92500.75 \\
500\_04 & \textbf{4281106062.5} & 4281107260.0 & \textbf{4281106062.5} & 4281140384.54 & 105528.32 \\
500\_05 & 3675293000.5 & 3675369229.0 & \textbf{3675292982.5} & 3675300412.71 & 16773.97 \\
   \hline
\end{tabular}
\caption{Comparison of objective function values for \texttt{Anjos-large} instances with current state of the art and robustness check}
  \label{tab:anjos_comb_results}
\end{table}

In table~\ref{tab:anjos_comb_results}, the best solution value for each instance is highlighted in bold.  Our window-based matheuristic matches or improves upon the best known results
for all \texttt{Anjos-large} instances, and establishes new best known upper bounds in two cases, demonstrating its competitiveness.

A closer look reveals that for instance sizes up to 400 facilities, our approach consistently matches the objective values reported by \citet{tang2022solving}.
This strongly suggests that these solutions may already be optimal or very close to optimal. When considering the largest instances with 500 facilities,
our method discovers improved upper bounds in two out of five cases, further underlining its effectiveness in tackling large-scale SRFLP instances.

As discussed in subsection \ref{sec: impr_wind}, these improvements can be attributed to the novel neighborhood structure embedded within our window-based approach.
It is also important to note that the window mechanism not only improves these two specific instances, but contributes significantly to the overall performance.
Without it, the algorithm would not be able to match the quality of solutions found by the AI-guided $KMPG$ metaheuristic.

As shown in the last column of table~\ref{tab:anjos_comb_results}, our window-based matheuristic achieves a very low standard deviation across the 15 independent runs for each instance,
highlighting the robustness of our approach. To assess consistency, we now look at the relative percentual standard deviation with respect to the average objective value.
The largest relative standard deviation across all instances is below $0.0037 \%$, while the average relative standard deviation is below $0.0009 \%$. 
These extremely small variations indicate highly consistent performance across independent runs.
Remarkably, for six out of the 20 instances, the standard deviation is exactly zero, meaning that the algorithm found the same best objective value in all 15 independent runs.

\paragraph{Computational results for \texttt{Sko-large} instance set} \label{subsec: res_sko}
Next, we evaluate the performance of our window-based matheuristic on the \texttt{Sko-large} benchmark instances. Again, we compare our results against those
of the $KMPG$ algorithm \citep{tang2022solving}, which has achieved the best known objective values for all \texttt{Sko-large} instances.
Additionally, we include $GRASP_F$ \citep{cravo2019grasp}, the previous state-of-the-art metaheuristic for this benchmark.
As in the previous comparison, we use the best values obtained by these algorithms as reference points. Table~\ref{tab:sko_comb_results} presents the results:
As before, the first column indicates instance size. Columns two and three list the best known values from $KMPG$ and $GRASP_F$.
The final three columns report the best, average, and standard deviation of our results across 15 runs.

\begin{table}[ht]
\centering
\begin{tabular}{lrr|rrr}
  \hline
  & $KMPG$ & $GRASP_F$ & \multicolumn{3}{c}{window\_appr} \\
  instance & & & BEST & AVG & SD \\ 
  \hline
  200\_01 & \textbf{3230706.0} & 3231044.0 & \textbf{3230706.0} & 3230712.93 & 14.30 \\
  200\_02 & \textbf{7758927.0} & 7758934.0 & \textbf{7758927.0} & 7758927.00 & 0.00 \\
  200\_03 & \textbf{12739043.0} & \textbf{12739043.0} & \textbf{12739043.0} & 12739175.07 & 143.47 \\
  200\_04 & \textbf{20260531.0} & \textbf{20260531.0} & \textbf{20260531.0} & 20260551.24 & 103.65 \\
  200\_05 & \textbf{26871976.5} & 26871990.5 & \textbf{26871976.5} & 26872065.11 & 298.67 \\
  300\_01 & \textbf{11249787.0} & 11252008.0 & \textbf{11249787.0} & 11249811.17 & 21.87 \\
  300\_02 & \textbf{28991767.0} & 28991854.0 & \textbf{28991767.0} & 28991823.93 & 211.40 \\
  300\_03 & \textbf{48790881.5} & 48791025.5 & \textbf{48790881.5} & 48790914.36 & 120.64 \\
  300\_04 & \textbf{71185096.5} & 71185259.5 & \textbf{71185096.5} & 71185188.09 & 375.75 \\
  300\_05 & \textbf{86789519.5} & 86789543.5 & \textbf{86789519.5} & 86789898.43 & 233.99 \\
  400\_01 & 26696039.0 & 26708999.0 & \textbf{26696000.0} & 26696072.88 & 36.62 \\
  400\_02 & 67950486.0 & 67950673.0 & \textbf{67950476.0} & 67951937.96 & 3706.02 \\
  400\_03 & \textbf{115881368.0} & 115881934.0 & \textbf{115881368.0} & 115881645.59 & 401.40 \\
  400\_04 & \textbf{162932778.0} & 162933020.0 & \textbf{162932778.0} & 162932780.52 & 6.67 \\
  400\_05 & \textbf{227632499.5} & 227633018.5 & \textbf{227632499.5} & 227632499.50 & 0.00 \\
  500\_01 & 52853174.0 & 52873391.0 & \textbf{52853051.0} & 52853354.88 & 217.84 \\
  500\_02 & 127362138.0 & 127364274.0 & \textbf{127362131.0} & 127362139.08 & 14.49 \\
  500\_03 & 230877719.5 & 230919458.5 & \textbf{230877706.5} & 230877794.30 & 437.75 \\
  500\_04 & \textbf{340273141.0} & 340274982.0 & \textbf{340273141.0} & 340279654.78 & 8637.99 \\
  500\_05 & 445698570.0 & 445699477.0 & \textbf{445698560.0} & 445698652.67 & 213.98 \\
   \hline
\end{tabular}
\caption{Comparison of objective function values for \texttt{Sko-large} instances with current state of the art and robustness check}
\label{tab:sko_comb_results}
\end{table}

Figure~\ref{tab:sko_comb_results} highlights the best solution values for each instance in bold. As seen with the \texttt{Anjos-large} dataset, our window approach consistently
reaches at least the current best known solutions for all 20 instances in the \texttt{Sko-large} set and even discovers a new upper bounds for six instances.
For problems up to 300 facilities, our results match those of $KMPG$, suggesting these solutions may be optimal. For 400- and 500-facility instances, we improve upon six out of ten instances and
notably, four of the five largest instances yield new best known results.
These improvements are once again attributable to the new neighborhood observed by the novel window approach as can be seen in section \ref{sec: impr_wind}.

The last column of table~\ref{tab:sko_comb_results} shows the standard deviation of the objective values across the 15 independent runs for each instance.
The average relative standard deviation across all instances is below $0.0006\%$, with the largest relative standard deviation being below $0.006\%$.
These small variations indicate that our approach is highly consistent across runs. Notably, for two out of the 20 instances, the standard deviation is exactly zero, indicating that the best objective
value was consistently found in every single run.

\paragraph{Computational results for \texttt{Palubeckis-large} instance set} \label{subsec: res_palu}
Finally, we assess our matheuristic on the \texttt{Palubeckis-large}\sloppy\ benchmark set, which includes 30 instances ranging from 310 to 1000 facilities.
As before, we compare against the $KMPG$ algorithm~\citep{tang2022solving} and the $GRASP_F$ metaheuristic~\citep{cravo2019grasp}. Table~\ref{tab:palu_results} summarizes the results.
The instance name and size are indicated in the first column. The second and third columns list the best known results from $KMPG$ and $GRASP_F$, respectively.
The next two columns show the best and average results from our method, and the final column provides the standard deviation across 15 runs.

\begin{table}[ht]
\centering
\begin{tabular}{lrr|rrr}
  \hline
  & $KMPG$ & $GRASP_F$ & \multicolumn{3}{c}{window\_appr} \\
instance & & & BEST & AVG & SD \\ 
  \hline
310 & \textbf{105754955.0} & 105754993.0 & \textbf{105754955.0} & 105755011.14 & 134.45 \\ 
320 & \textbf{119522881.5} & 119522963.5 & \textbf{119522881.5} & 119523172.39 & 475.79 \\ 
330 & \textbf{124891823.5} & \textbf{124891823.5} & \textbf{124891823.5} & 124891823.50 & 0.00 \\ 
340 & \textbf{129796777.5} & 129796900.5 & \textbf{129796777.5} & 129796794.06 & 50.72 \\ 
350 & \textbf{149594388.0} & \textbf{149594388.0} & \textbf{149594388.0} & 149594393.22 & 15.05 \\ 
360 & \textbf{141122187.0} & 141122226.0 & \textbf{141122187.0} & 141122733.24 & 1472.93 \\ 
370 & \textbf{174663159.5} & 174663982.5 & \textbf{174663159.5} & 174663159.50 & 0.00 \\ 
380 & \textbf{189452403.5} & \textbf{189452403.5} & \textbf{189452403.5} & 189463165.85 & 8860.82 \\ 
390 & \textbf{208688557.0} & 208689095.0 & \textbf{208688557.0} & 208688557.00 & 0.00 \\ 
400 & \textbf{213768810.5} & 213768843.5 & \textbf{213768810.5} & 213768870.41 & 281.00 \\ 
410 & \textbf{243494269.0} & 243494328.0 & \textbf{243494269.0} & 243494641.72 & 369.69 \\ 
420 & \textbf{270756527.5} & \textbf{270756527.5} & \textbf{270756527.5} & 270756535.42 & 10.78 \\ 
430 & 286335267.5 & 286335891.5 & \textbf{286334521.5} & 286348075.92 & 6669.74 \\ 
440 & \textbf{301067264.5} & 301067654.5 & \textbf{301067264.5} & 301067273.31 & 19.39 \\ 
450 & \textbf{324488485.0} & 324488523.0 & \textbf{324488485.0} & 324489369.88 & 2049.17 \\ 
460 & \textbf{314884659.0} & 314885454.0 & \textbf{314884659.0} & 314884706.11 & 169.74 \\ 
470 & \textbf{379529990.0} & 379534882.0 & \textbf{379529990.0} & 379534565.81 & 1799.83 \\ 
480 & \textbf{366821074.0} & 366821213.0 & \textbf{366821074.0} & 366822075.00 & 1682.25 \\ 
490 & \textbf{413901954.5} & 413920947.5 & \textbf{413901954.5} & 413902025.31 & 117.69 \\ 
500 & \textbf{465570835.5} & 465578299.5 & \textbf{465570835.5} & 465578278.85 & 9470.54 \\ 
550 & \textbf{587090450.5} & 587090690.5 & \textbf{587090450.5} & 587092832.21 & 2812.25 \\ 
600 & 801567648.5 & 801570093.5 & \textbf{801567640.5} & 801593075.15 & 12851.39 \\ 
650 & 927512856.0 & 927519028.0 & \textbf{927512834.0} & 927526107.44 & 24563.40 \\ 
700 & 1158462254.0 & 1158462486.0 & \textbf{1158462242.0} & 1158523472.89 & 75030.83 \\ 
750 & \textbf{1438408836.5} & 1438627405.5 & \textbf{1438408836.5} & 1438417330.25 & 9244.16 \\ 
800 & 1861593391.0 & 1861722621.0 & \textbf{1861593361.0} & 1861628131.79 & 39080.83 \\ 
850 & 2126675923.0 & 2126880783.0 & \textbf{2126675899.0} & 2126726619.76 & 71482.59 \\ 
900 & 2600003561.0 & 2600420560.0 & \textbf{2600003520.0} & 2600132911.58 & 85329.60 \\ 
950 & 2993153592.5 & 2993414032.0 & \textbf{2993153552.5} & 2993175682.09 & 50930.53 \\ 
1000 & 3426646969.0 & 3426810369.0 & \textbf{3426646843.0} & 3426743204.14 & 158984.85 \\ 
   \hline
\end{tabular}
\caption{Comparison of objective function values for \texttt{Palubeckis-large} instances with current state of the art and robustness check}
\label{tab:palu_results}
\end{table}

Our approach successfully matches or improves the best known solutions in all 30 instances, improving upon 9 of them, including several of the largest problems. 
This demonstrates its ability to scale effectively to very large instances.
As in previous experiments, the method is highly consistent. The largest relative standard deviation across all runs is below $0.0065\%$, while for three instances, the standard deviation is exactly zero,
confirming the robustness of the window approach matheuristic.

\section{Conclusion and outlook}
\label{sec:conclusion}

In this work, we introduce a novel matheuristic approach for solving the 
\emph{single row facility layout problem (SRFLP)}.
The proposed method is built upon a new concept we denote as window approach, 
which uses the betweenness-based \emph{mixed-integer programming (MIP)}
formulation of 
the SRFLP.
This approach maintains the global facility ordering fixed except within a selected moving window, allowing the relative positioning of facilities inside the window to be optimized exactly
while respecting the fixed permutation outside. This targeted partial optimization leads to the definition of a new and powerful neighborhood structure,
which enables the algorithm to outperform existing state-of-the-art metaheuristics.
The MIP-based component is integrated into a broader matheuristic framework, which 
begins with a \emph{multi-start simulated annealing (MSA)}
metaheuristic that generates high-quality and diverse initial solutions. These solutions are subsequently refined using two complementary local search procedures,
both of which rely on the same efficient move evaluation formulas used for the MSA phase. This shared structure ensures fast evaluations,
especially in the later stages when fewer improving moves remain. 
The two local search strategies are combined with the window approach to explore a neighborhood that, to the best of our knowledge,
has not been investigated before in the SRFLP literature.
To assess the performance of the method, we conducted extensive experiments using 
the widely adopted \texttt{Anjos-large, Sko-large} and \texttt{Palubeckis-large} 
benchmark sets.
Our algorithm improves the best-known solution values for 17 
of 70 instances, and matches the best-known solution values for the remaining 53 instances. The results demonstrate the effectiveness of the window approach, particularly for larger instances with at least 500 facilities.

Beyond the SRFLP, our proposed window approach could have potential 
applications in a broader class of facility layout 
problems and other problems involving permutations where betweeness-based MIP 
formulations perform well, such as the minimum linear arrangement problem 
\citep{caprara2011optimal}, or problems which allow for asymmetric 
betweeness-based MIP such as 
discussed in \citet{mallach2023binary}. An 
integration of the window approach with other metaheuristics can also be an 
interesting topic for further research.

\section*{Acknowledgments}
This research was funded in whole, or in part, by the Austrian Science Fund 
(FWF)[P 35160-N]. For the purpose of open access, the author has applied a CC BY 
public copyright licence to any Author Accepted Manuscript version arising from 
this submission.
It is also supported by the Johannes Kepler University Linz, Institute of Technology (Project LIT-2021-10-YOU-216).
This research was also funded in part by the Austrian Science Fund (FWF) 10.55776/COE12.

\bibliographystyle{elsarticle-harv}

\bibliography{references}

\newpage
\appendix

\section{Proof: Transformation of the objective function}
\label{app:proof}

\begin{proof}[Proof: Transformation of the objective function]
    For any solution of the SRFLP, denoted as $\pi$, we can express the objective function using the betweenness-based MIP formulation as follows:

\begin{equation}
    \sum_{i=1}^{n-1} \sum_{j=i+1}^{n} w_{\pi_i, \pi_j} \left( \frac{\ell_{\pi_i} + \ell_{\pi_j}}{2} + \sum_{k=i+1}^{j-1} \ell_{\pi_k} \right)
    \label{eq:start_proof}
\end{equation}
Next, we use the indices $sw$ and $ew$, representing the positions of the first and last facility inside the window, respectively.
We then reorganize the summation indices in the objective function to ensure all facility pairs are still considered:

\begin{equation}
    \begin{split}
        \sum_{i=1}^{n-1} \sum_{j=i+1}^{n} w_{\pi_i, \pi_j} \left( \frac{\ell_{\pi_i} + \ell_{\pi_j}}{2} + \sum_{k=i+1}^{j-1} \ell_{\pi_k} \right) = \\
        \sum_{i=1}^{sw-2} \sum_{j=i+1}^{sw-1} w_{\pi_i, \pi_j} \left( \frac{\ell_{\pi_i} + \ell_{\pi_j}}{2} + \sum_{k=i+1}^{j-1} \ell_{\pi_k} \right) + \\
        \sum_{i=1}^{sw-1} \sum_{j=sw}^{ew} w_{\pi_i, \pi_j} \left( \frac{\ell_{\pi_i} + \ell_{\pi_j}}{2} + \sum_{k=i+1}^{j-1} \ell_{\pi_k} \right) + \\
        \sum_{i=1}^{sw-1} \sum_{j=ew+1}^{n} w_{\pi_i, \pi_j} \left( \frac{\ell_{\pi_i} + \ell_{\pi_j}}{2} + \sum_{k=i+1}^{j-1} \ell_{\pi_k} \right) + \\
        \sum_{i=sw}^{ew-1} \sum_{j=i+1}^{ew} w_{\pi_i, \pi_j} \left( \frac{\ell_{\pi_i} + \ell_{\pi_j}}{2} + \sum_{k=i+1}^{j-1} \ell_{\pi_k} \right) + \\
        \sum_{i=sw}^{ew} \sum_{j=ew+1}^{n} w_{\pi_i, \pi_j} \left( \frac{\ell_{\pi_i} + \ell_{\pi_j}}{2} + \sum_{k=i+1}^{j-1} \ell_{\pi_k} \right) + \\
        \sum_{i=ew+1}^{n-1} \sum_{j=i+1}^{n} w_{\pi_i, \pi_j} \left( \frac{\ell_{\pi_i} + \ell_{\pi_j}}{2} + \sum_{k=i+1}^{j-1} \ell_{\pi_k} \right)
    \end{split}
    \label{eq:proof_seperate_indexes}
\end{equation}

Examining equation \eqref{eq:proof_seperate_indexes}, the first and last summands are entirely independent of the permutation inside the window.
The fourth summand, on the other hand, is fully dependent on the permutation inside the window. Additionally, by introducing $\ell_w = \sum_{i = sw}^{ew} \ell_i$, we can rewrite the third summand as:

\begin{equation}
    \begin{split}
        \sum_{i=1}^{sw-1} \sum_{j=ew+1}^{n} w_{\pi_i, \pi_j} \left( \frac{\ell_{\pi_i} + \ell_{\pi_j}}{2} + \sum_{k=i+1}^{j-1} \ell_{\pi_k} \right) = \\
        \sum_{i=1}^{sw-1} \sum_{j=ew+1}^{n} w_{\pi_i, \pi_j} \left( \frac{\ell_{\pi_i} + \ell_{\pi_j}}{2} + \sum_{k=i+1}^{sw-1} \ell_{\pi_k} + \ell_w + \sum_{k=ew+1}^{j-1} \ell_{\pi_k} \right)
    \end{split}
    \label{eq:split_l_3}
\end{equation}

Since the total length of the facilities inside the window $\ell_w$, does not depend on the permutation inside the window, it follows that this summand remains independent of the permutation inside the window.
Furthermore, summands two and five still depend on the permutation within the window. To analyze them further, we decompose these distances in the following equations.
First, we split the second summand of equation \eqref{eq:proof_seperate_indexes}:
\begin{equation}
    \begin{split}
        \sum_{i=1}^{sw-1} \sum_{j=sw}^{ew} w_{\pi_i, \pi_j} \left( \frac{\ell_{\pi_i} + \ell_{\pi_j}}{2} + \sum_{k=i+1}^{j-1} \ell_{\pi_k} \right) = \\
        \sum_{i=1}^{sw-1} \sum_{j=sw}^{ew} w_{\pi_i, \pi_j} \left( \frac{\ell_{\pi_i}}{2} + \sum_{k=i+1}^{sw-1} \ell_{\pi_k} \right) + \\
        \sum_{i=1}^{sw-1} \sum_{j=sw}^{ew} w_{\pi_j, \pi_i} \left( \frac{\ell_{\pi_j}}{2} + \sum_{k=sw}^{j-1} \ell_{\pi_k} \right)
    \end{split}
    \label{eq:split_l_2}
\end{equation}
From equation \eqref{eq:split_l_2}, it is now clear that the first term in the decomposed sum is independent of the permutation within the window.
We than also seperate the fifth summand of equation \eqref{eq:proof_seperate_indexes}:
\begin{equation}
    \begin{split}
        \sum_{i=sw}^{ew} \sum_{j=ew+1}^{n} w_{\pi_i, \pi_j} \left( \frac{\ell_{\pi_i} + \ell_{\pi_j}}{2} + \sum_{k=i+1}^{j-1} \ell_{\pi_k} \right) = \\
        \sum_{i=sw}^{ew} \sum_{j=ew+1}^{n} w_{\pi_i, \pi_j} \left( \frac{\ell_{\pi_i}}{2} + \sum_{k=i+1}^{ew} \ell_{\pi_k} \right) + \\
        \sum_{i=sw}^{ew} \sum_{j=ew+1}^{n} w_{\pi_j, \pi_i} \left( \frac{\ell_{\pi_j}}{2} + \sum_{k=ew+1}^{j-1} \ell_{\pi_k} \right)
    \end{split}
    \label{eq:split_l_5}
\end{equation}

Similar from equation \eqref{eq:split_l_5}, it is now also clear that the second term in the decomposed sum is independent of the permutation within the window.
Substituting equations \eqref{eq:split_l_3} - \eqref{eq:split_l_5} into \eqref{eq:proof_seperate_indexes}, we obtain:
\begin{equation}
    \begin{split}
        \sum_{i=1}^{sw-2} \sum_{j=i+1}^{sw-1} w_{\pi_i, \pi_j} \left( \frac{\ell_{\pi_i} + \ell_{\pi_j}}{2} + \sum_{k=i+1}^{j-1} \ell_{\pi_k} \right) + \\
        \sum_{i=1}^{sw-1} \sum_{j=sw}^{ew} w_{\pi_i, \pi_j} \left( \frac{\ell_{\pi_i}}{2} + \sum_{k=i+1}^{sw-1} \ell_{\pi_k} \right) + \\
        \sum_{i=1}^{sw-1} \sum_{j=sw}^{ew} w_{\pi_j, \pi_i} \left( \frac{\ell_{\pi_j}}{2} + \sum_{k=sw}^{j-1} \ell_{\pi_k} \right) + \\
        \sum_{i=1}^{sw-1} \sum_{j=ew+1}^{n} w_{\pi_i, \pi_j} \left( \frac{\ell_{\pi_i} + \ell_{\pi_j}}{2} + \sum_{k=i+1}^{sw-1} \ell_{\pi_k} + \ell_w + \sum_{k=ew+1}^{j-1} \ell_{\pi_k} \right) + \\
        \sum_{i=sw}^{ew-1} \sum_{j=i+1}^{ew} w_{\pi_i, \pi_j} \left( \frac{\ell_{\pi_i} + \ell_{\pi_j}}{2} + \sum_{k=i+1}^{j-1} \ell_{\pi_k} \right) + \\
        \sum_{i=sw}^{ew} \sum_{j=ew+1}^{n} w_{\pi_i, \pi_j} \left( \frac{\ell_{\pi_i}}{2} + \sum_{k=i+1}^{ew} \ell_{\pi_k} \right) + \\
        \sum_{i=sw}^{ew} \sum_{j=ew+1}^{n} w_{\pi_j, \pi_i} \left( \frac{\ell_{\pi_j}}{2} + \sum_{k=ew+1}^{j-1} \ell_{\pi_k} \right) + \\
        \sum_{i=ew+1}^{n-1} \sum_{j=i+1}^{n} w_{\pi_i, \pi_j} \left( \frac{\ell_{\pi_i} + \ell_{\pi_j}}{2} + \sum_{k=i+1}^{j-1} \ell_{\pi_k} \right)
    \end{split}
    \label{eq:split_obj}
\end{equation}
In equation \eqref{eq:split_obj}, the weighted distances have now been separated into components that are either dependent or independent of the permutation within the window.
Consequently, in the next step, we can aggregate all weighted distances that are independent of the permutation, as they can be precomputed in advance and therefore do not need to be considered
when solving the MIP model using the window approach. The aggregation of the fixed weighted distances is shown in equation~\ref{eq:fixed}.


In the final step we set in $fixed(sw, ew, \pi)$ into equation \eqref{eq:split_obj}:
\begin{equation}
    \begin{split}
        \text{fixed(sw, ew, } \pi\text{)} + \sum_{i=1}^{sw-1} \sum_{j=sw}^{ew} w_{\pi_j, \pi_i} \left( \frac{\ell_{\pi_j}}{2} + \sum_{k=sw}^{j-1} \ell_{\pi_k} \right) + \\
        \sum_{i=sw}^{ew-1} \sum_{j=i+1}^{ew} w_{\pi_i, \pi_j} \left( \frac{\ell_{\pi_i} + \ell_{\pi_j}}{2} + \sum_{k=i+1}^{j-1} \ell_{\pi_k} \right) + \\
        \sum_{i=sw}^{ew} \sum_{j=ew+1}^{n} w_{\pi_i, \pi_j} \left( \frac{\ell_{\pi_i}}{2} + \sum_{k=i+1}^{ew} \ell_{\pi_k} \right) = \\
        \sum_{i=1}^{n-1} \sum_{j=i+1}^{n} w_{\pi_i, \pi_j} \left( \frac{\ell_{\pi_i} + \ell_{\pi_j}}{2} + \sum_{k=i+1}^{j-1} \ell_{\pi_k} \right),
    \end{split}
    \label{eq:final_obj}
\end{equation}

where $fixed(sw, ew, \pi)$ represents the precomputed independent components for the current window location in the current solution.
\end{proof}
This transformation confirms that optimizing the permutation inside the window only requires considering the dependent terms,
significantly reducing computational complexity in the MIP model. The distinction between independent and dependent weighted distances is visually illustrated in figure \ref{fig:graph_srflp_all},
where black lines represent independent distances and green lines denote distances influenced by the permutation inside the window.

\section{Additional pseudocodes}
\label{app:appendixB}

\subsection{Implemention of the MSA}
\label{subsec:MSA}
The main procudure of the MSA, which not only follows the implementation shown by \citet{palubeckis2017single} is implemented in
algorithm \ref{alg:MSA}. Note that in contrast to \citep{palubeckis2017single} our MSA implementation contains the \texttt{LS\_wind} algorithm.

In our MSA algorithm, the SA method is repeatedly executed with randomly generated starting solutions until the stopping criterion is met.
The starting solutions are random permutations of the facilities, as this approach is recommended by \citet{palubeckis2017single}.
The core structure of the SA algorithm consists of two nested for loops, commonly referred to as the "outer" and "inner" loops. The outer for loop
controls the temperature while the inner loop generates moves a fixed number of times $\hat{z} \cdot n$ for each temperature level.
In each iteration of the outer loop the temperature gets cooled down by the factor $\alpha$. Before we enter the outer for loop we first need to compute the maximum temperature $T_{max}$ for each starting solution. 
To compute $T_{max}$ we compute the maximum absolute move gain for a swap of facilities of the random starting solution.
We than fix $T_{min}$ to 0.0001 and get the number of outer loops, which is computed as follows: $\beta' = (log(T_{min}) - log(T_{max}))/log(\alpha)$ \citep{palubeckis2017single}.
The most important parameter, which manages the probability of a swap move, is set to $p$. Therefore the probability for an insertion move is equal 1 - $p$.
With these fixed parameters we can already enter the inner loop for each temperature, starting with $T = T_{max}$. 
First the move type is selected randomly. 
Then we compute the move gain $\delta$ for the selected move type.
The detailed implementations of these functions are shown by \citet{palubeckis2017single}.
These functions return a gain value $\delta$. In case $\delta < 0$ the move is accepted, otherwise the probability for the acceptance is computed as shown in \citet{palubeckis2017single}.
It can also be observed that at the beginning of the SA approach, when the temperature is still high, the probability of accepting a move that increases the objective value,
and thus promotes diversification, is relatively high.
To ensure time efficience in computing and performing the moves we set the parameter $\gamma$, which controls until which percentage of the cooling parameter the first computation method is used, equal to 0.35.
As explained in \citet{palubeckis2017single} this parameter decides on which computation method to use for the given $\beta$ value.
If the moves are accepted we apply the selected move using the detailed implementation shown by \citet{palubeckis2017single}.
A more detailed explanation of the adapted MSA algorithm for the window approach can be found in \citet{pammer2025matheuristic}.
If the objective value obtained after a SA phase is lower a treshhold, which is dependent on the current best objective value, then the current solution enters $LS\_wind$ where the solution is further improved using the LS methods in combination with the innovative window approach.
In this setup, the window size vector $wsv$ is set to (13), and the Boolean parameter $gb$ is set to \texttt{false}, meaning that the local search algorithm is not restarted after the first improvement found
within the window approach. This configuration is chosen to enhance time efficiency.
The SA procedure is repeated until the predefined time limit is reached, and it returns the best overall solution along with its corresponding objective value.

\begin{algorithm}[ht]
\caption{\texttt{MSA}}
\label{alg:MSA}
\KwIn{$\alpha$, $\hat{z}$, $p$, $T_{min}$, $\gamma$, $wsv$, $gb$}
\KwOut{$\pi^*$, $F(\pi^*)$}
generate a random starting solution $\pi$\;
$\pi^{'}, \pi^{*} \gets \pi$\;
$T_{max} \gets \texttt{compute\_T\_max}(\pi)$\;
$\beta' \gets \texttt{compute\_beta\_prime}(T_{max}, T_{min}, \alpha)$\;
\While{$comp\_time < max\_comp\_time$}{
    $T \gets T_{max}$\;
    \For{$\beta \gets 1$ \KwTo $\beta'$}{
        \For{$z \gets 1$ \KwTo $\hat{z} \cdot n$}{
            generate $rn \sim \mathcal{U}(0, 1)$\;
            \eIf{$rn < p$}{
                randomly select facilities $r$ and $s$\;
                generate move gain $\delta$ of swapping facility r and s\;
            }{
                randomly select facility $r$ and position $l$\;
                generate move gain $\delta$ of inserting facility r in position l\;
            }
            \eIf{$\delta < 0$}{
                $accept$ $\gets$ true\;
            }{
                generate $acc\_prob \sim \mathcal{U}(0, 1)$\;
                \eIf{$acc\_prob < \exp(-\delta/T)$}{
                    $accept \gets$ true\;
                }{
                    $accept \gets$ false\;
                }
            }
            \If{$accept$}{
                $F(\pi^{'}) \gets F(\pi^{'}) + \delta$\;
                \eIf{$rn < p$}{
                    swap positions of facility $r$ and $s$ in the solution $\pi^{'}$\;
                }{
                    insert facility $r$ on position $l$ in the solution $\pi^{'}$\;
                }
                \If{$F(\pi^{'}) < F(\pi)$}{
                    $\pi \gets \pi^{'}$\;
                }
            }
        }
        $T \gets \alpha \cdot T$\;
    }
    \If{$F(\pi) < F(\pi^{*}) \cdot 1.00001$}{
        $\pi \gets \texttt{LS\_wind}(\pi, F(\pi), wsv, gb)$\;
    }
    \If{$F(\pi) < F(\pi^{*})$}{
        $\pi^* \gets \pi$\;
    }
    generate a new random starting solution $\pi$\;
}
\Return{$\pi^*$, $F(\pi^{*})$}\;
\end{algorithm}
\clearpage

\subsection{Create order of three dimensional decision variables}
\label{subsec: create_order}

The algorithm \texttt{create\_order} (see algorithm~\ref{alg:order}) determines the order of facilities within a window, based on the three-dimensional decision variable $x_{i,k,j}$ and the window size $ws$.

\begin{algorithm}[h!t]
\caption{\texttt{create\_order}}
\label{alg:order}
\KwIn{$x_{i,k,j}$, $ws$}
\KwOut{$order$}

\For{$i \gets 1$ \KwTo $ws - 1$}{
    \For{$j \gets i+1$ \KwTo $ws$}{
        $b_{i,j} \gets \sum_{\substack{k=1 \\ k \ne i \\ k \ne j}}^{ws} x_{i,k,j}$\;
    }
}
$s, e \gets \arg\max(b)$\;
$order[s] \gets 1$\;

\For{$j \gets 1$ \KwTo $s - 1$}{
    $order[j] \gets \sum_{\substack{k=1 \\ k \ne if \\ k \ne j}}^{ws} x_{j,k,s} + 2$\;
}
\For{$j \gets s + 1$ \KwTo $ws$}{
    $order[j] \gets \sum_{\substack{k=1 \\ k \ne if \\ k \ne j}}^{ws} x_{s,k,j} + 2$\;
}
\Return{$order$}\;
\end{algorithm}

In the first step, this function calculates the number of facilities that lie between each pair of facilities. The underlying logic is that the pair with the highest number of intermediate facilities
must be located at the borders of the window.
Therefore, in the next step, we identify and assign the pair with the most facilities between them to the variables $s$ and $e$, representing the two endpoints of the ordering.
Subsequently, the $s^{\text{th}}$ entry is set to one, indicating that the facility originally located at position $s$ is now assigned to the first position in the window.
The remaining ordering is then determined by examining how many facilities lie between the first facility and each of the others. More precisely, the second facility in the sequence is defined
as the one with zero facilities in between itself and the first facility.
This logic is extended by assigning an ordering index equal to the number of facilities between the first and any other facility, plus two, to account for the position shift.
The resulting vector provides the order of the facilities within the window. However, it is important to note that $s$ could correspond to either the start or the end of the true ordering.
For this reason, the window approach must also include a check for the direction of the permutation to ensure consistency with the overall solution structure.

\subsection{Compute fixed weighted distances}
\label{subsec: fix_w_dist}

\begin{algorithm}[h!t]
\caption{\texttt{calc\_out\_w\_dist}}
\label{alg:outwdist}
\KwIn{$\pi$, $sw$, $ew$}
\KwOut{$wd$}
$wb_{wind} \gets$ cumulative weights for facilities before the window to all facilities inside the window\;
\For{$i \gets 1$ \KwTo $sw - 2$}{
    \For{$j \gets i+1$ \KwTo $sw - 1$}{
        $wd \gets wd + w[\pi_i, \pi_j] \cdot \left( \frac{\ell_{\pi_i} + \ell_{\pi_j}}{2} + \ell_{betw}[i] \right)$\;
        $\ell_{betw}[i] \gets \ell_{betw}[i] + \ell_{\pi_j}$\;
    }
    $wd \gets wd + wb_{wind}[i] \cdot \left( \ell_{betw}[i] + \frac{\ell_{\pi_i}}{2} \right)$\;
}
$wd \gets wd + wb_{wind}[sw - 1] \cdot \frac{\ell_{\pi_{sw - 1}}}{2}$\;

$\ell_{wind} \gets \sum_{i=sw}^{ew} \ell_{\pi_i}$\;

\For{$i \gets 1$ \KwTo $sw - 1$}{
    $\ell_{aft} \gets 0$\;
    \For{$j \gets ew + 1$ \KwTo $n$}{
        $wd \gets wd + w[\pi_i, \pi_j] \cdot \left( \frac{\ell_{\pi_i} + \ell_{\pi_j}}{2} + \ell_{wind} + \ell_{aft} + \ell_{betw}[i] \right)$\;
        $\ell_{aft} \gets \ell_{aft} + \ell_{\pi_j}$\;
    }
}

$wa_{wind} \gets$ cumulative weights for facilities after the window to all facilities inside the window\;

$\ell_{aft} \gets \sum_{i=ew+1}^{n} \ell_{\pi_i}$\;

\For{$i \gets ew+1$ \KwTo $n-1$}{
    \For{$j \gets i+1$ \KwTo $n$}{
        $wd \gets wd + w[\pi_i, \pi_j] \cdot \left( \frac{\ell_{\pi_i} + \ell_{\pi_j}}{2} + \ell_{betw}[i-ew] \right)$\;
        $\ell_{betw}[i-ew] \gets \ell_{betw}[i-ew] + \ell_{\pi_j}$\;
    }
    $wd \gets wd + wa_{wind}[i-ew] \cdot \left( \ell_{aft} - \frac{\ell_{\pi_i}}{2} - \ell_{betw}[i-ew] \right)$\;
}

$wd \gets wd + wa_{wind}[n - ew] \cdot \left( \ell_{aft} - \frac{\ell_{\pi_n}}{2} - \ell_{betw}[n - ew] \right)$\;

\Return{$wd$}\;
\end{algorithm}

The \texttt{calc\_out\_w\_dist} algorithm (see algorithm~\ref{alg:outwdist}) begins by computing the cumulative weights from each facility before the window to all facilities inside the window.
Subsequently, all weighted distances between pairs of facilities located before the window are added.
To ensure time-efficient computation, we use the vector $\ell_{betw}$, which stores the cumulative lengths of the facilities involved.

Next, the weighted distances between the facilities before the window and the facilities inside the window, for the part of the distance up to the starting point of the window, are added.
Before computing the distances between facilities before and after the window, the total length of all facilities inside the window is calculated. This value is required as it contributes to the distance between each pair of facilities before and after the window.

The weighted distances between pairs of facilities before and after the window are then computed and added to the total sum $wd$. These are calculated also using the previously computed $\ell_{betw}$ vector to maintain computational efficiency.
To further optimize this step, we introduce the variable $\ell_{aft}$, which stores the cumulative lengths of facilities after the window.

In the final step, the function computes the weighted distances between all pairs of facilities after the window, as well as the partial weighted distances from the ending point of the window to the facilities inside it.
This involves computing the cumulative weights from each facility after the window to all facilities inside the window, along with the total length of the facilities after the window, again stored in $\ell_{aft}$.

Finally, the function returns the total fixed weighted distance of the current solution.

\end{document}